\documentclass[12pt]{amsart}
\usepackage{amssymb}
\usepackage{amsmath}
\usepackage{amsfonts}
\usepackage{tikz-cd}
\usepackage{stmaryrd}
\usepackage{hyperref}
\newtheorem{theorem}{Theorem}[section]
\newtheorem{lemma}[theorem]{Lemma}
\newtheorem{proposition}[theorem]{Proposition}
\newtheorem{corollary}[theorem]{Corollary}
\newtheorem{definition}[theorem]{Definition}

\theoremstyle{definition}
\newtheorem{remark}[theorem]{Remark}
\newtheorem{example}[theorem]{Example}

\numberwithin{equation}{section}
\allowdisplaybreaks[1]

\newcommand{\A}{\mathbb{A}}
\newcommand{\Z}{\mathbb{Z}}

\newcommand{\KMW}{\mathrm{K}^\mathrm{MW}}
\newcommand{\KM}{\mathrm{K}^\mathrm{M}}

\newcommand{\tbb}[1]{\widetilde{\mathbb{#1}}}
\newcommand{\wt}[1]{\widetilde{#1}}
\newcommand{\Spec}{\mathrm{Spec}\ }
\newcommand{\af}{\mathbb{A}}
\newcommand{\afnz}[1]{\mathbb{A}^{#1}\setminus \{0\}}

\newcommand{\bZ}{\mathbb{Z}}
\newcommand{\bfZ}{\mathbf{Z}}
\newcommand{\tbZ}{\tbb{Z}}
\newcommand{\Gm}{\mathbb{G}_m}

\newcommand{\Daba}[1]{D(Ab_{\af^1}(#1))}
\newcommand{\DM}{\mathrm{DM}}
\newcommand{\DMt}{\widetilde{\mathrm{DM}}}
\newcommand{\M}{\mathrm{M}}
\newcommand{\Mt}{\wt{\mathrm{M}}}
\newcommand{\Meta}{\wt{\mathrm{M}}_{\eta}}

\newcommand{\aBra}[1]{\left<#1\right>}
\newcommand{\llBra}[1]{\llbracket #1 \rrbracket
}

\newcommand{\Ker}{\mathrm{Ker}}

\newcommand{\xr}[1]{\xrightarrow{#1}}

\newcommand{\AZ}{\mathbb{A}\mathcal{Z}}
\newcommand{\bfZa}{\mathbf{Z}_{\mathbb{A}^1}}
\newcommand{\tauO}{\tau^{\odot}}

\title{Cellular $\mathbb{A}^1$-Homology of Smooth Toric Varieties}
\author{Haoyang Liu, Keyao Peng}
\address{}
\email{}
\thanks{}

\begin{document}
\begin{abstract}
In this paper, we present the calculations of cellular $\mathbb{A}^1$-homology for smooth toric varieties, along with an explicit description of pure shellable cases. Consequently, we derive the (Milnor-Witt) motivic decomposition for these pure shellable cases. Furthermore, we obtain an additive basis for the Chow groups of general smooth toric varieties.
\end{abstract}
\maketitle
\tableofcontents
\section{Introduction}
Toric varieties represent a significant class of algebraic varieties, entirely characterized by fans, which are forms of combinatorial data. In the smooth case, a fan can be represented as $\Sigma = (K, \lambda)$, where it consists of two components: a simplicial complex $ K $ with $ m $ vertices, and a linear morphism $ \lambda: \mathbb{Z}^m \to \mathbb{Z}^n $.

A simplicial complex $ K $ on a finite set $ V $ is defined as a collection of subsets of $ V $ satisfying the following conditions:
\begin{enumerate}
  \item If $ v \in V $, then $ \{v\} \in K $.
  \item If $ \sigma \in K $ and $ \tau \subset \sigma $, then $ \tau \in K $.
\end{enumerate}

For a given toric variety $ X_\Sigma $, the complex points $ X_\Sigma(\mathbb{C}) $ and the real points $ X_\Sigma(\mathbb{R}) $ are topologically distinct spaces and exhibit differing properties in their cohomology groups. In general, the cohomology groups of the real points are more intricate compared to those of the complex points \cite{cai2021integral}.

Recent advancements in motivic $ \af^1 $-homotopy \cite{Mor10}\cite{bachmann2020milnor} have introduced new frameworks for understanding algebraic varieties through homotopy theory. Notably, these advancements enable the recovery of the homotopy types of both the complex and real points \cite{bachmann2018motivic}. A significant algebraic invariant associated with these varieties is the Chow group $ \mathrm{CH}^*(X_\Sigma) $, which gives rise to the cycle class map to $ H^*(X_\Sigma(\mathbb{C}), \mathbb{Z}) $ and $ H^*(X_\Sigma(\mathbb{R}), \mathbb{Z}/2) $. An emerging development within algebraic geometry is the Chow-Witt group $ \wt{\mathrm{CH}}^*(X_\Sigma) $, which serves as a finer invariant and includes a real class map to $ H^*(X_\Sigma(\mathbb{R}), \mathbb{Z}) $ \cite{hornbostel2021real}.

Inspired by \cite{cai2021integral}, this work aims to compute the cellular $ \af^1 $-homology of smooth toric varieties. The concept of cellular $ \af^1 $-homology, defined in \cite{MOREL2023109346}, is particularly useful for explicitly calculating the (co)homology in $ \af^1 $-homotopy. This concept parallels the role of cellular homology in classical topology, facilitating computations of various cohomology theories, including motivic cohomology, Chow groups, MW-motivic cohomology, and Chow-Witt groups. The cellular $ \af^1 $-chain complex $ C^{cell}_*(X_\Sigma) $ can be constructed by gluing oriented ``cubes" along their faces. Consequently, when considering rational motives in $ \DM(k)_{\mathbb{Q}} $, it follows that $ \M(X_\Sigma)_{\mathbb{Q}} $ is represented as the direct sum of several $ \mathbb{Q}(q)[p] $ (see Corollary \ref{rationDecomp}). To achieve explicit computations, it is imperative to comprehend the orientation of gluing and the boundary map.

A notable advantage of representing a fan as $ \Sigma = (K, \lambda) $ is that a smooth toric variety $X_\Sigma$ (or more generally, simplicial varieties) can be identified the quotient $ \AZ_K / \Ker(\exp(\lambda)) $ \cite[Theorem 2.1]{cox1995homogeneous}, where $ \AZ_K $ is a subvariety of $ \mathbb{A}^m $. All cells and boundaries of $ C^{cell}_*(\AZ_K) $ are well oriented. Additionally, the exponential map $ \exp(\lambda): \Gm^m \to \Gm^n $ is induced from the morphism $ \lambda $.

We can approach the problem of determining orientation by examining the action of the toric group $\Ker(\exp(\lambda))$ on cubical cells of $\AZ_K$, as detailed in Proposition \ref{tActMain}. We define a complex $C^{can}_*(\AZ_K)$ that consists of canonical cells, which include at most one point from each orbit of the $\Ker(\exp(\lambda))$ action. It follows that $C^{can}_*(\AZ_K) \cong C^{cell}_*(X_\Sigma)$ (see Proposition \ref{canIso}). Moreover, we can introduce the subcomplex $\overline{C}^{can}_*(\AZ_K) \subset C^{can}_*(\AZ_K)$, which serves as a retraction (refer to Proposition \ref{restrQis}). The complex $\overline{C}^{can}_*(\AZ_K)$ can be constructed from combinatorial data through manual calculations.

When focusing on \textbf{pure shellable} toric varieties, we can achieve more explicit computations.  Consider an order on maximal subsets $ K_{\max} = \{\sigma_1, \ldots, \sigma_g\} \subset K $. We say $K$ is pure of dimension $d$, if we have $\forall i,\ |\sigma_i|=d$, and a fan $\Sigma = (K, \lambda) $ is pure, if $K$ is pure of dimension $n$. Then we define the sets of subsets $ \min(\sigma_i) = \{ \tau \subset \sigma_i \mid \tau \text{ is minimal such that } \tau \not\subset \sigma_j \text{ for all } j < i\} $. In the cases that $K$ is shellable, there exists an order such $ \min(\sigma_i)$ consists of a single subset $r(\sigma_i) \subset \sigma_i$.

Let $\mathrm{row}\lambda$ denote the mod-$2$ row set associated with $\lambda$ (see definition in Section \ref{sec:shellable}). For a fixed $\lambda$, let $K_\omega$ represent a specific subcomplex of $K$ formed by intersecting with $\omega \in \mathrm{row}$. We define $G^{\lambda}_{i} = \bigoplus_{\omega \in \mathrm{row}\lambda}\widetilde{H}_{i}(|K_\omega|)$ as the direct sum of the reduced homology groups. The basis $B(0) \subset K_{\max}$ forms the free part of $G^{\lambda}_{*,free}$, while for $l>1$, the basis $B(l) \subset K_{\max}$ corresponds to the $l$-torsion part $G^{\lambda}_{*,l-tor}$. Similarly, we denote the bases for the groups $G^{\lambda}_{i,free/l-tor}$ as $B(l)_i$.

Let $\KMW_* \in \Daba{k}$ represent the Milnor-Witt K-theory, which plays an analogous role to $\mathbb{Z}$ in classical topology (as further defined later). For $l \in \mathbb{N}^+$, we define $ \KMW_i \sslash l \eta := [\KMW_{i+1} \xr{l \eta} \KMW_{i}] \in \Daba{k}$, and set $\KMW_i \sslash 0 = \KMW_i$. Consequently, we obtain the following decomposition:

\begin{theorem}[Corollary \ref{mainA1cellular}]\label{homology}

In the category $\Daba{k}$, for a smooth pure shellable toric variety $X_{\Sigma}$, we have a quasi-isomorphism:
\[
C^{cell}_*(X_\Sigma) \cong \bigoplus_{l \in \mathbb{N}} \bigoplus_{\sigma \in B(l)} \KMW_{|r(\sigma)|} \sslash l \eta [|r(\sigma)|],
\]
for a certain subset $B(1) \subset K_{\max}$, which can be derived from the complex $\bigoplus_{\omega \in \mathrm{row}\lambda} \overline{C}^{cri}_i(K_\omega)$.

Specifically, we establish the correspondences between:
\begin{itemize}
  \item the free summands of $G^{\lambda}_{i-1,free}$ and the free summands of $ \KMW_i[i] $.
  \item the torsion summands of $G^{\lambda}_{i-1,tor}$ and the torsion summands $ \KMW_i \sslash l \eta [i] $ for $l > 1$.
\end{itemize}
As a result, the cellular $\mathbb{A}^1$-homology of $X_\Sigma$ is given by
\[
\mathbf{H}^{cell}_i(X_{\Sigma}) = \bigoplus_{l \in \mathbb{N}} \bigoplus_{B(l)_{i-1}} \KMW_i / l \eta \oplus \bigoplus_{l \in \mathbb{N}^+} \bigoplus_{B(l)_{i-2}} (_{l\eta}\KMW_i).
\]
Moreover, we can establish correspondences between:
\begin{itemize}
  \item the free summands of $G^{\lambda}_{i-1,free}$ and the free summands of $ \KMW_i $.
  \item the torsion summands of $G^{\lambda}_{i-1,tor}$ and the torsion summands of $\KMW_i / l \eta $ for $l > 1$. 
  \item the torsion summands of $G^{\lambda}_{i-2,tor}$ and the torsion summands of ${_{\l \eta}\KMW_i}$ for $l > 1$. 
\end{itemize}
\end{theorem}

If $K$ is the boundary complex of a simple $n$-polytope (for example, if $X_\Sigma$ is projective), we can apply a similar argument as in \cite[Corollary 1.6]{cai2021integral} to conclude that when the dimension $n \leq 4$, the group $\mathbf{H}^{cell}_i(X_\Sigma)$ contains only $\eta$-torsion elements.

Recall that $\KMW_i / \eta \cong \KM_i$ and ${_\eta\KMW_i} \cong 2\KM_i$. Let $b_i = \mathrm{rk}(G^{\lambda}_{i, free})$ denote the $i$-th total reduced Betti number, and denote the number of vertices of $K$ as $m$. Additionally, let $X^d_\Sigma$ represent a space of dimension $d$. More precisely, we have the following results for various dimensions:

\[
\begin{array}{|c|c|c|} 
\hline 
\mathbf{H}^{cell}_i(X^2_\Sigma) & \text{Orientable} & \text{Non-orientable} \\ 
\hline 
i = 0 & \mathbb{Z} & \mathbb{Z} \\ 
i = 1 & (\KMW_1)^{m-2} & (\KMW_1)^{m-3} \oplus \KM_2 \\ 
i = 2 & \KMW_2 & 2\KM_2 \\ 
\hline 
\end{array} 
\]

\[
\begin{array}{|c|c|c|} 
\hline 
\mathbf{H}^{cell}_i(X^3_\Sigma) & \text{Orientable} & \text{Non-orientable} \\ 
\hline 
i = 0 & \mathbb{Z} & \mathbb{Z} \\ 
i = 1 & (\KMW_1)^{b_0} \oplus (\KM_1)^{m-3-b_0} & (\KMW_1)^{b_0} \oplus (\KM_1)^{m-3-b_0} \\ 
i = 2 & (\KMW_2)^{b_0} \oplus (2\KM_2)^{m-3-b_0} & (\KMW_2)^{b_0-1} \oplus \KM_2 \oplus (2\KM_2)^{m-3-b_0} \\ 
i = 3 & \KMW_3 & 2\KM_3 \\ 
\hline 
\end{array} 
\]

\[
\begin{array}{|c|c|c|} 
\hline 
\mathbf{H}^{cell}_i(X^4_\Sigma) & \text{Orientable} & \text{Non-orientable} \\ 
\hline 
i=0 & \bZ & \bZ \\ 
i=1 & (\KMW_1)^{b_0} \oplus (\KM_1)^{m-4-b_0} & (\KMW_1)^{b_0} \oplus (\KM_1)^{m-4-b_0} \\ 
i=2 & (\KMW_2)^{b_1} \oplus (\KM_2)^{m-4-b_0} \oplus (2\KM_2)^{m-4-b_0} & (\KMW_2)^{b_1} \oplus (\KM_2)^{m-5-b_2} \oplus (2\KM_2)^{m-4-b_0} \\ 
i=3 & (\KMW_3)^{b_0} \oplus (2\KM_3)^{m-4-b_0} & (\KMW_3)^{b_2} \oplus \KM_3 \oplus (2\KM_3)^{m-5-b_2} \\ 
i=4 & \KMW_4 & 2\KM_4 \\
\hline 
\end{array} 
\]

It is also worth mentioning that one can utilize Theorem \ref{homology} to determine the structure of cellular $\af^1$-homology groups of Coxeter toric varieties (See Remark \ref{Coxeter}).

By applying the functor $\mathbf{L}\tilde{\gamma}^*$ to the category $\DMt(k)$ \cite[\S 1.3 (MW4)]{bachmann2020milnor}, we can obtain an MW-motivic decomposition relative to the cones $ C(l\eta) = \tbZ \sslash l \eta $, as studied in \cite{fasel2023tate}.

\begin{corollary}
  
In the category $\DMt(k)$, for a smooth pure shellable toric variety $ X_{\Sigma} $, we have the following MW-motivic decomposition:
\[
  \Mt(X_\Sigma) \cong \bigoplus_{l \in \mathbb{N}} \bigoplus_{\sigma \in B(l)} \tbZ \sslash l \eta (|r(\sigma)|)[2|r(\sigma)|],
\]
where $ B(1) \subset K_{\max} $ is a subset that can be derived from the complex $ \bigoplus_{\omega \in \mathrm{row} \lambda} \overline{C}^{cri}_i(K_\omega) $. In particular, this establishes correspondences between:
\begin{itemize}
  \item the free summands of $ G^{\lambda}_{i-1, free} $ and the free summands $ \tbZ(i)[2i] $,
  \item the torsion summands of $ G^{\lambda}_{i-1, tor} $ and the torsion summands $ \tbZ \sslash l \eta (i)[2i] $ for $ l > 1 $.
\end{itemize}
\end{corollary}

When passing to the category where $ \eta $ is invertible, there are no terms from $ B(1) $, indicating that the decomposition depends solely on $ G^{\lambda}_{*} $.

\begin{corollary}
  
In the category $ \DMt(k)[\eta^{-1}] $, for a smooth pure shellable toric variety $ X_{\Sigma} $, we can express the decomposition as follows:
\[
  \Meta(X_\Sigma) \cong \bigoplus_{1 \neq l \in \mathbb{N}} \bigoplus_{\sigma \in B(l)} \tbZ \sslash l (|r(\sigma)|)[2|r(\sigma)|],
\]
\end{corollary}

Upon considering the category where $ \eta = 0 $, the complex can be simplified, allowing for a motivic decomposition via the aforementioned quasi-isomorphism. The following corollary serves as a generalization of \cite[\S 5.2, Theorem p.102-104]{fulton1993introduction} concerning motives.

\begin{corollary}

In the category $\DM(k)$, for a smooth pure shellable toric variety $ X_{\Sigma} $, we obtain the following motivic decomposition:

\[
  \M(X_\Sigma) \cong \bigoplus_{\sigma \in K_{\max}} \bZ( |r(\sigma)| )[ 2|r(\sigma)| ]
\]

The generators are represented by $ [e^{r( \sigma )}] $, corresponding to the regular expanding sequence of $ K $.

In particular, we have the decomposition of motivic cohomology (and, consequently, the Chow group):

\[
  \mathrm{H}^{*,*}_{\mathrm{M}}(X_\Sigma) \cong \bigoplus_{\sigma \in K_{\max}} \mathrm{H}^{*-2|r(\sigma)|,*-|r(\sigma)|}_{\mathrm{M}}(k) [e^{r( \sigma )}]
\]

\end{corollary}

For smooth fans that are not pure or not shellable, as discussed in Section \ref{exotic}, the results presented in Corollary \ref{mainA1cellular} do not hold. However, if we focus solely on the Chow group, we can still derive a basis. The following proposition can be viewed as a generalization of \cite[\S 5.2, Theorem p.102-104]{fulton1993introduction} for non-complete cases.

\begin{proposition}[Proposition \ref{chowDecomp}]
  For a smooth toric variety $ X_{\Sigma} $, consider an order on $ K_{\max} = \{\sigma_1, \ldots, \sigma_g\} $. Recall the definition of the sets $ \min(\sigma_i) = \{ \tau \subset \sigma_i \mid \tau \text{ is minimal such that } \tau \not\subset \sigma_j \text{ for all } j < i\} $. We then have the following decomposition of the Chow group:

\[
  \mathrm{CH}^*(X_\Sigma) \cong \bigoplus_{\sigma \in K_{\max}} \bigoplus_{\tau \in \min(\sigma)} \bZ [e^{\tau}]
\]

The generators are given by $ [e^{\tau}] \in \mathrm{CH}^{|\tau|}(X_\Sigma) $.
\end{proposition}

\subsection*{Acknowledgement}
The authors express their gratitude to Fangzhou Jin for hosting them at Tongji University, which provided the opportunity to initiate this work.
\subsection*{Notations}

Let $ k $ denote a field. We refer to the big Nisnevich site of smooth schemes of finite type over $ k $ as $ Sm_k $, and we denote the category of simplicial Nisnevich sheaves of sets over $ Sm_k $ by $ \Delta^{op}Shv_{Nis}(Sm_k) $. This latter category is also termed the \emph{category of spaces}. The category of spaces possesses a \emph{Nisnevich local injective model structure}.

The associated homotopy category is referred to as the \emph{simplicial homotopy category}, denoted by $ \mathcal{H}_s(k) $. The left Bousfield localization of the Nisnevich local injective model structure with respect to the collection of all projection morphisms $ \mathcal{X} \times \A^1 \to \mathcal{X} $ (as $ \mathcal{X} $ varies over all simplicial sheaves) is termed the $ \A^1 $-model structure. The corresponding homotopy category is called the $ \A^1 $-homotopy category and is denoted by $ \mathcal{H}(k) $. Additionally, there exists a pointed version of the construction of $ \mathcal{H}(k) $, for which we begin by considering the category of pointed spaces, whose objects are pairs $(\mathcal{X}, x)$, where $ x: \Spec k \to \mathcal{X} $ serves as a base point, and we end up with the pointed $\A^1$-homotopy category, denoted as $\mathcal{H}_{\bullet}(k)$.

For a field $ k $, we denote the Milnor-Witt K-theory by $\KMW_*(k) = \bigoplus_{n \in \mathbb{Z}} \KMW_n(k)$. This is the associated graded ring generated by a symbol $\eta$ of degree -1 and symbols $[u]$ of degree 1 for each $ u \in k^\times $. The relations governing these symbols are as follows: $[u][1-u] = 0$, $[uv] = [u] + [v] + \eta [u][v]$, $\eta[u] = [u]\eta$, and $\eta(2 + \eta[-1]) = 0$ for any $ u, v \in k^\times $. Additionally, we denote the symbol $\langle u \rangle$ as $1 + \eta [u] \in \KMW_0(k) = \text{GW}(k)$, and $h := 1 + \aBra{-1} = 2 + \eta[-1]$.

Let $\mathrm{K}^\mathrm{M}$ be the $ n $-th unramified Milnor K-theory sheaf and $\KMW$ be the $ n $-th unramified Milnor-Witt K-theory sheaf, as discussed in \cite{Mor10}. The sheaves $\mathrm{K}^\mathrm{M}$ and $\KMW$ are characterized as strictly $\A^1$-invariant sheaves in the sense defined in \cite{Mor10} over any field.

The character function $ \chi: \mathbb{Z} \to \{0,1\} $ is defined by $ n \mapsto i $ for $ i \equiv n \mod 2 $. Let $ \llBra{n,m} := \{i\in \bZ \mid n \leq i \leq m \} $ and $ \llBra{n} := \llBra{1, n} $.

\section{Cellular $\af^1$-homology}
We proceed by revisiting the construction and basic properties of the cellular $\A^1$-homology developed in \cite{MOREL2023109346}.

\begin{definition}
    A Nisnevich sheaf $\mathcal{F}$ of abelian groups on $ Sm_k $ is termed \emph{$\A^1$-invariant} if the projection map $ U \times \A^1 \to U $ induces a bijection
    \begin{equation*}
      \mathcal{F}(U) \to \mathcal{F}(U \times \A^1)
    \end{equation*}
    for every $ U \in Sm_k $. Furthermore, $\mathcal{F}$ is called \emph{strictly $\A^1$-invariant} if for every integer $ i \geq 0 $, the projection map $ U \times \A^1 \to U $ induces a bijection
    \begin{equation*}
      H^i_{Nis}(U, \mathcal{F}) \to H^i_{Nis}(U \times \A^1, \mathcal{F})
    \end{equation*}
    for every $ U \in Sm_k $.
\end{definition}

We denote the abelian category of Nisnevich sheaves of abelian groups on $ Sm_k $ as $ Ab(k) $ and the category of strictly $\A^1$-invariant sheaves of abelian groups as $ Ab_{\A^1}(k) $. Notably, the latter category is indeed an abelian category and possesses a symmetric monoidal structure, represented by $\otimes$ \cite{Mor05}.

In our work with complexes of Nisnevich sheaves of abelian groups on $ Sm_k $, we adhere to standard homological conventions. Let $ Ch_{\geq 0}(Ab(k)) $ denote the category of chain complexes $ C_* $ consisting of objects in $ Ab(k) $ with differentials of degree -1, such that $ C_n = 0 $ for all $ n < 0 $. The normalized chain complex functor $ C_*: \Delta^{op}Shv_{Nis}(Sm_k) \to Ch_{\geq 0}(Ab(k)) $ induces a functor $ C_*: \mathcal{H}_s(k) \to D(Ab(k)) $. Upon $\A^1$-localization, this yields a functor $ C^{\A^1}_*: \mathcal{H}(k) \to D_{\A^1}(k) $, where $ D_{\A^1}(k) $ is the full subcategory of $ D(Ab(k)) $ consisting of $\A^1$-local complexes. The inclusion of $ D_{\A^1}(k) $ into $ D(Ab(k)) $ admits a left adjoint, specifically the $\A^1$-localization functor $ L_{\A^1} $. Thus, for a space $ \mathcal{X} \in \Delta^{op}Shv_{Nis}(Sm_k) $, one has $ C^{\A^1}_*(\mathcal{X}) = L_{\A^1}(C_*(\mathcal{X})) $. If $ \mathcal{X} $ is a pointed space, the reduced chain complex $ \tilde{C}_*(\mathcal{X}) $ is defined as the kernel of the canonical morphism $ C_*(\mathcal{X}) \to \mathbb{Z} $. Furthermore, there exists a canonical morphism $ C_*(\mathcal{X}) \cong \mathbb{Z} \oplus \tilde{C}_*(\mathcal{X}) $. The same argument holds after taking $ \A^1 $-localization, allowing us to define the reduced $ \A^1 $-chain complex as $ \tilde{C}^{\A^1}_*(\mathcal{X}) = L_{\A^1}(\tilde{C}_*(\mathcal{X})) $.

\begin{definition}
   For any $ \mathcal{X} \in \Delta^{op}Shv_{Nis}(Sm_k) $ and $ n \in \mathbb{Z} $, the $ n $-th $ \A^1 $-homology sheaf of $ \mathcal{X} $ is defined to be the $ n $-th homology sheaf of the $ \A^1 $-chain complex $ C^{\A^1}_*(\mathcal{X}) $. If $ \mathcal{X} $ is pointed, we define the $ n $-th reduced $ \A^1 $-homology sheaf of $ \mathcal{X} $ as $ \tilde{\mathbf{H}}^{\A^1}_n (\mathcal{X}) := \mathbf{H}^{\A^1}_n (\tilde{C}^{\A^1}_*(\mathcal{X})) $.
\end{definition}

Since $ \mathbf{H}^{\A^1}_0 (\Spec k) \cong \mathbb{Z} $ and $ \mathbf{H}^{\A^1}_n (\Spec k) = 0 $ for $ n \neq 0 $, we have an isomorphism
\[
  \mathbf{H}^{\A^1}_* (\mathcal{X}) \cong \mathbb{Z} \oplus \tilde{\mathbf{H}}^{\A^1}_* (\mathcal{X})
\]
of graded abelian sheaves. For any pointed space $ \mathcal{X} $ and any $ n \in \mathbb{Z} $, there exists a canonical isomorphism
\[
  \tilde{\mathbf{H}}^{\A^1}_n (\mathcal{X}) \cong \tilde{\mathbf{H}}^{\A^1}_{n+1} (S^1 \wedge \mathcal{X})
\]
as the $ \A^1 $-localization functor commutes with the simplicial suspension functor in $ D(Ab(k)) $. Additionally, by the $ \A^1 $-connectivity theorem \cite{Mor10}, it is known that for every space $ \mathcal{X} $ and every integer $ n $, the $ \A^1 $-homology sheaves $ \mathbf{H}^{\A^1}_n (\mathcal{X}) $ are strictly $ \A^1 $-invariant sheaves that vanish if $ n < 0 $.

For any Nisnevich sheaf of sets $ \mathcal{F} $ on $ Sm_k $, we define 
\[
  \mathbf{Z}_{\mathbb{A}^1}[\mathcal{F}] := \mathbf{H}^{\A^1}_0 (\mathcal{F}).
\]
For a pointed Nisnevich sheaf of sets $ \mathcal{F} $ on $ Sm_k $, we define 
\[
  \mathbf{Z}_{\mathbb{A}^1}(\mathcal{F}) := \tilde{\mathbf{H}}^{\A^1}_0 (\mathcal{F}).
\]
The sheaf $ \mathbf{Z}_{\mathbb{A}^1}[\mathcal{F}] $ is regarded as the \emph{free strictly $ \A^1 $-invariant sheaf} on $ \mathcal{F} $; specifically, it satisfies the property that there exists a canonical bijection 

\[
\text{Hom}_{Ab(k)}(\bfZ_{\A^1}[\mathcal{F}], M)\stackrel{\sim}{\to} \text{Hom}_{{Shv_{Nis}}(Sm_k)}(\mathcal{F}, M)
\]
for every $ M \in Ab(k) $. A similar bijection also holds for pointed sheaves.

\begin{remark}\cite{Mor10}
  With the above notations, the canonical morphism $ \Z((\Gm)^{\wedge n}) \to \KMW_n $ induces an isomorphism 
  \[
  \bfZ_{\A^1}(\Gm^{\wedge n}) \cong \KMW_n.
  \]
\end{remark}

If $ M $ is a strictly $ \A^1 $-invariant sheaf, there exists a bijection
\[
H^n_{Nis}(\mathcal{X}, M) \stackrel{\sim}{\to} \text{Hom}_{D(Ab(k))}(C_*^{\A^1}(\mathcal{X}), M[n]).
\]

Next, we introduce the concept of a cellular structure for $ X \in Sm_k $ and discuss orientations of a vector bundle on a smooth $ k $-scheme.

\begin{definition}
    An object $ X \in Sm_k $ is said to be \emph{cohomologically trivial} if the cohomology groups $ H^n_{Nis}(X, M) $ vanish for all positive degrees and for every strictly $ \mathbb{A}^1 $-invariant sheaf $ M \in Ab_{\mathbb{A}^1}(k) $.
\end{definition}

\begin{remark}
    By definition, $ X $ is cohomologically trivial if and only if the sheaf $ \bfZ_{\mathbb{A}^1}[X] = \mathbf{H}^{\A^1}_0 (X) $ is a projective object in the abelian category $ Ab_{\A^1}(k) $. Examples of cohomologically trivial schemes include $ \mathbb{A}^1 $, $ \mathbb{G}_m $, and, more generally, open subschemes of $ \mathbb{A}^1 $. Moreover, it is straightforward to demonstrate that the product of two cohomologically trivial smooth $ k $-schemes is also cohomologically trivial.
\end{remark}

With the preceding definitions, we can now present the formal definition of a cellular structure:

\begin{definition}\label{cell}
    An $ n $-dimensional smooth scheme $ X $ possesses a \emph{cellular structure} if it admits an increasing filtration by open subschemes
    \[
    \emptyset = \Omega_{-1}(X) \subset \Omega_0(X) \subset \Omega_1(X) \subset \cdots \subset \Omega_{s-1}(X) \subset \Omega_s(X) = X,
    \]
    such that for each $ i \in \{0, \ldots, s\} $, the reduced induced closed subscheme $ \Omega_i(X) \setminus \Omega_{i-1}(X) $ is a finite disjoint union of schemes that are $ k $-smooth, affine, have codimension $ i $, and are cohomologically trivial. A scheme that has a cellular structure is referred to as a \emph{cellular scheme}.
\end{definition}

\begin{example}\label{projective}
    The projective space $\mathbb{P}^n$ possesses a cellular structure defined by a full flag $\mathbb{P}^0\subset\mathbb{P}^1\subset\cdots\subset\mathbb{P}^n$, with open sets given by $\Omega_i(\mathbb{P}^n) := \mathbb{P}^n \setminus \mathbb{P}^{n-i-1}$. The relationship $\Omega_i(\mathbb{P}^n) \setminus \Omega_{i-1}(\mathbb{P}^n) = \mathbb{P}^{n-i} \setminus \mathbb{P}^{n-i-1} \cong \mathbb{A}^{n-i}$ establishes that the sequence of open subschemes
    \[
    \emptyset = \Omega_{-1}(\mathbb{P}^n) \subset \Omega_0(\mathbb{P}^n) \subset \Omega_1(\mathbb{P}^n) \subset \cdots \subset \Omega_n(\mathbb{P}^n) = \mathbb{P}^n
    \]
    provides a cellular structure on $\mathbb{P}^n$.
\end{example}

We now proceed to define an oriented cellular structure on a smooth $k$-scheme $X$.

\begin{definition}
    Let $ r \geq 1 $, and let $ \xi $ be a rank $ r $ vector bundle over a smooth $ k $-scheme $ X $. We say that $ \xi $ is \emph{orientable} if the line bundle $ \wedge^r(\xi) $ is trivial. An orientation of $ \xi $ is defined as a choice of an isomorphism $ \theta: \mathbb{A}^1_X \cong \wedge^r(\xi) $, where $ \mathbb{A}^1_X $ denotes the trivial vector bundle of rank 1 over $ X $. We denote by $ Or(\xi) $ the set of orientations of the vector bundle $ \xi $.
\end{definition}

\begin{remark}\label{trivial}
    It is noteworthy that a rank $ r $ vector bundle $ \xi $ over a smooth, cohomologically trivial $ k $-scheme $ X $ is inherently orientable. This result follows from the fact that the line bundle $ \wedge^r(\xi) $ corresponds, up to isomorphism, to an element in the group $ H^1(X, \mathbb{G}_m) $, which is trivial.
\end{remark}

For instance, a trivialization of the vector bundle $ \xi $ over $ X $ induces an orientation of $ \xi $. Given $ \theta \in Or(\xi) $ and a unit $ \alpha \in \mathcal{O}(X)^{\times} $, the expression $ \alpha \cdot \theta $ also represents an orientation of $ \xi $. Consequently, the group of units $ \mathcal{O}(X)^{\times} $ acts naturally on $ Or(\xi) $.

\begin{definition}\label{orientedcell}
    Let $ X $ be a smooth $ k $-scheme. An \emph{oriented cellular structure} on $ X $ consists of a cellular structure on $ X $,
    \[
    \emptyset = \Omega_{-1}(X) \subset \Omega_0(X) \subset \Omega_1(X) \subset \cdots \subset \Omega_{s-1}(X) \subset \Omega_s(X) = X,
    \]
    together with an orientation $ o_i $ of the normal bundle $ \nu_i $ of the closed immersion $ X_i := \Omega_i \setminus \Omega_{i-1} \subset \Omega_i $ for each $ i $. The collection of choices of the orientations $ o_i $ constitutes an orientation of the given cellular structure.
\end{definition}

\begin{example}\label{orientedproj}
    Recall the cellular structure of the $ n $-dimensional projective space $ \mathbb{P}^n $ over $ k $ as detailed in Example \ref{projective}. If $ X_0, \ldots, X_n $ denote the homogeneous coordinates on $ \mathbb{P}^n $, then we have $ \Omega_i(\mathbb{P}^n) = \bigcup_{j=0}^i D_+(X_j) $, where $ D_+(X_j) $ is the complement of the hyperplane in $\mathbb{P}^n$ defined by $X_j$. Thus, we have that 

\[\Omega_i(\mathbb{P}^n)\setminus \Omega_{i-1}(\mathbb{P}^n)=\mathbb{P}^{n-i}\setminus\mathbb{P}^{n-i-1}\] 

is the zero locus on $\Omega_i(\mathbb{P}^n)$ of the rational functions 

\[\frac{X_0}{X_i}, \ldots, \frac{X_{i-1}}{X_i}.\]

Since this sequence forms a regular sequence, the normal bundle $\nu_i$ is trivialized by this sequence, which induces an orientation. Consequently, this provides us with a canonical oriented cellular structure on $\mathbb{P}^n$.

\end{example}

\begin{remark}
   As established in \cite[Lemma 2.13]{MOREL2023109346}, the normal bundles $\nu_i$ in Definition \ref{orientedcell} of an oriented cellular structure on a smooth $k$-scheme are always trivial. The emphasis is not solely on orientability, but rather on the choice of a specific orientation.
\end{remark}

In particular, we have the Thom isomorphism theorem for oriented bundles.

\begin{lemma}[{\cite[Lemma 2.23]{MOREL2023109346}}]\label{thom}
  Let $\xi$ be a trivial vector bundle of rank $r$ over a smooth $k$-scheme $X$ equipped with an orientation $o \in Or(\xi)$. There exists an isomorphism $\phi:\xi \cong \mathbb{A}^r_X$ of vector bundles over $X$, which preserves the orientations. Furthermore, if $X$ is cohomologically trivial, then the induced pointed 
  $\mathbb{A}^1$-homotopy class 

\[
Th(\xi) \cong_{\mathbb{A}^1} \mathbb{A}^r/(\mathbb{A}^r\setminus\{0\}) \wedge X_+ 
\] 

depends only on $o \in Or(\xi)/(\mathcal{O}(X)^{\times})^{\wedge 2}$.
\end{lemma}

Next, we will introduce the notions of cellular $\mathbb{A}^1$-chain complexes, oriented cellular $\mathbb{A}^1$-chain complexes, and the corresponding cellular $\mathbb{A}^1$-homology. Let $X$ be a smooth 
$k$-scheme with a cellular structure given by 

\[
\emptyset=\Omega_{-1}(X) \subset \Omega_0(X) \subset \Omega_1(X) \subset \cdots \subset \Omega_{s-1}(X) \subset \Omega_s(X)=X
\] 

as defined in Definition \ref{cell}. Define 

\[
\Omega_i(X)\setminus \Omega_{i-1}(X)=\coprod_{j\in J_i} Y_{ij}
\] 

to be the decomposition of $\Omega_i(X)\setminus \Omega_{i-1}(X)$ into irreducible components. By applying the motivic homotopy purity theorem \cite{MV99}, we obtain a canonical pointed $\mathbb{A}^1$-weak equivalence 

\[
\Omega_i(X)/\Omega_{i-1}(X) \cong_{\mathbb{A}^1} Th(\nu_i), 
\] 

where $\nu_i$ denotes the normal bundle of the closed immersion $X_i \hookrightarrow \Omega_i(X)$ of smooth $k$-schemes.

Now, consider the cofibration sequence 

\[
\Omega_{i-1}(X) \to \Omega_i(X) \to \Omega_i(X)/\Omega_{i-1}(X).
\]

In the long exact sequence of reduced $\mathbb{A}^1$-homology sheaves, we have 

\[
\cdots \to \tilde{\mathbf{H}}^{\mathbb{A}^1}_n (\Omega_{i-1}(X)) \to \tilde{\mathbf{H}}^{\mathbb{A}^1}_n (\Omega_i(X)) \to \tilde{\mathbf{H}}^{\mathbb{A}^1}_n (\Omega_i(X)/\Omega_{i-1}(X)) \stackrel{\delta}{\to} \tilde{\mathbf{H}}^{\mathbb{A}^1}_{n-1} (\Omega_{i-1}(X)) \cdots
\]

The purity theorem provides us with the identification 

\[
\tilde{\mathbf{H}}^{\mathbb{A}^1}_n (\Omega_i(X)/\Omega_{i-1}(X)) \cong \tilde{\mathbf{H}}^{\mathbb{A}^1}_n (Th(\nu_i)). 
\]

We define $\partial_n$ to be the composite

\begin{equation*}
  \tilde{\mathbf{H}}^{\A^1}_n (Th(\nu_n)) \stackrel{\delta}{\to} \tilde{\mathbf{H}}^{\A^1}_{n-1} (\Omega_{n-1}(X)) \to \tilde{\mathbf{H}}^{\A^1}_{n-1} (Th(\nu_{n-1}))
\end{equation*}
where the second morphism is induced by the cofibration sequence previously mentioned. By definition, we have $\partial_{n-1} \circ \partial_n = 0$ for every $ n $, which leads to the establishment of a chain complex in $ Ab_{\A^1}(k) $:
\begin{equation*}
  C_*^{cell}(X) := (\tilde{\mathbf{H}}^{\A^1}_n (Th(\nu_n)), \partial_n)_n.
\end{equation*}
The chain complex $ C_*^{cell}(X) $ is referred to as the cellular $\A^1$-chain complex associated with the cellular structure on $ X $, and its homology sheaves, denoted $ H_n(C^{cell}_*(X)) $, are termed the cellular $\A^1$-homology sheaves of $ X $.

In the scenario where $ X $ possesses an oriented cellular structure, the selected orientations, in conjunction with the pointed $\A^1$-equivalence as established in Lemma \ref{thom} and the canonical $\A^1$-weak equivalence $ \A^i / (\A^i \setminus \{0\}) \cong_{\A^1} S^i \wedge \Gm^{\wedge i} $, yield canonical pointed $\A^1$-weak equivalences:
\begin{equation*}
  \Omega_i(X) / \Omega_{i-1}(X) \cong_{\A^1} Th(\nu_i) \cong_{\A^1} S^i \wedge \Gm^{\wedge i} \wedge \left(\vee_{j \in J_i} (Y_{ij})_+\right).
\end{equation*}
Transitioning to the corresponding reduced $\A^1$-homology sheaves, we obtain the following identification for each $ n $:
\begin{equation*}
  \tilde{\mathbf{H}}^{\A^1}_n \left( \Omega_n(X) / \Omega_{n-1}(X) \right) \cong \tilde{\mathbf{H}}^{\A^1}_n (Th(\nu_n)) \cong \bigoplus_{j \in J_n} \tilde{\mathbf{H}}^{\A^1}_n (S^n \wedge \Gm^{\wedge n} \wedge (Y_{nj})_+).
\end{equation*}
Utilizing the suspension isomorphism of $\A^1$-homology alongside the identity $ \tilde{\mathbf{H}}^{\A^1}_0 (\Gm^{\wedge n}) = \mathbf{Z}_{\mathbb{A}^1}(\Gm^{\wedge n}) = \KMW_n $, we derive isomorphisms of strictly $\A^1$-invariant sheaves:
\begin{equation*}
  \bigoplus_{j \in J_n} \tilde{\mathbf{H}}^{\A^1}_n (S^n \wedge \Gm^{\wedge n} \wedge (Y_{nj})_+) \cong \bigoplus_{j \in J_n} \tilde{\mathbf{H}}^{\A^1}_0 (\Gm^{\wedge n} \wedge (Y_{nj})_+) \cong \bigoplus_{j \in J_n} \KMW_n \otimes \mathbf{Z}_{\mathbb{A}^1}[Y_{nj}].
\end{equation*}
In this context, the tensor product $ \otimes $ is understood to occur within $ Ab_{\A^1}(k) $. With the orientation considerations, the cellular chain complex $ C^{cell}_*(X) $ thus becomes isomorphic to the chain complex whose terms are represented by $ \bigoplus_{j \in J_n} \KMW_n \otimes \mathbf{Z}_{\mathbb{A}^1}[Y_{nj}] $ in degree $ n $.

Let us denote $ \tilde{\partial}_n $ as the composite:

\begin{equation*}
\oplus_{j\in J_n} \KMW_{n}\otimes\bfZ_{\mathbb{A}^1}[Y_{nj}] \cong \tilde{\mathbf{H}}^{\A^1}_n (Th(\nu_n)) \stackrel{\partial}{\to} \tilde{\mathbf{H}}^{\A^1}_{n-1} (Th(\nu_{n-1})) \cong \oplus_{j'\in J_{n-1}} \KMW_{n-1} \otimes \bfZ_{\mathbb{A}^1}[Y_{n-1, j'}]
\end{equation*}

\begin{definition}
    The chain complex, which resides in the category of strictly $\A^1$-invariant sheaves $Ab_{\A^1}(k)$, is defined as follows:
    \begin{equation*}
        \tilde{C}^{cell}_*(X) := \left( \oplus_{j\in J_n} \KMW_{n}\otimes \bfZ_{\mathbb{A}^1}[Y_{nj}], \tilde{\partial_n}\right)_n.
    \end{equation*}
    This complex is referred to as the oriented cellular $\A^1$-chain complex associated with the oriented cellular structure on $X$. The $n$-th homology sheaf of this complex will be designated as the $n$-th oriented cellular $\A^1$-homology of $X$. After selecting the orientations, it can be demonstrated that the $n$-th oriented cellular $\A^1$-homology of $X$ is isomorphic to the cellular $\A^1$-homology sheaves $\mathbf{H}^{cell}_n(X)$.
\end{definition}

\begin{remark}
    Let $X$ be a cellular smooth $k$-scheme. In the absence of a choice of orientations, the cellular $\A^1$-chain complex $C^{cell}_*(X)$ becomes nearly uncomputable. On the other hand, the oriented version, $\tilde{C}^{cell}_*(X)$, which depends on the specified orientations, is often more amenable to explicit computation. We will leverage this as a robust tool to advance our calculations.
\end{remark}

Given a sheaf of abelian groups $M$ on $Sm_k$, we can construct the cohomological cellular cochain complex $C^*_{cell}(X; M)$, which represents the (cochain) complex of abelian groups with terms defined by  
$C^n_{cell}(X;M):=\text{Hom}_{Ab_{\mathbb{A}^1}(k)}(C^{cell}_n(X), M)$, accompanied by the induced differentials.

\begin{proposition}[{\cite[Prop 2.27]{MOREL2023109346}}]
\label{cellCohomology}
    Let $X$ be a cellular smooth scheme over $k$. For any strictly $\mathbb{A}^1$-invariant sheaf $M$ on $Sm_k$, we have the following natural isomorphisms:
    \begin{equation*}
        H^n_{Nis}(X, M) \stackrel{\sim}{\to} H^n(C^*_{cell}(X; M)) \stackrel{\sim}{\to} \text{Hom}_{D(Ab_{\mathbb{A}^1}(k))}(C^{cell}_*(X), M[n]).
    \end{equation*}
\end{proposition}
  
\begin{remark}
    It is established in \cite{Mor10} that for any strictly $\A^1$-invariant sheaf $M$ on $Sm_k$ and any smooth $k$-scheme $X$, the comparison isomorphism $ H^n_{Zar}(X, M) = H^n_{Nis}(X, M) $ holds as an isomorphism for all $ n \geq 0 $. Based on Proposition \ref{cellCohomology}, we can use cellular $\A^1$-chain complexes to compute specific sheaf cohomology groups of a smooth cellular $k$-scheme $X$. Notably, we can take $ M $ to be $ \KMW $ in order to calculate the Chow-Witt groups or, more generally, the Milnor-Witt motivic cohomology groups for smooth cellular schemes, including varieties such as toric varieties.
\end{remark}

With the cellular structure defined in Example \ref{orientedproj} for projective space $ \mathbb{P}^n $, the following computation is known:

\begin{corollary}{\cite[Corollary 2.51]{MOREL2023109346}}
For any integer $ n \geq 1 $, the oriented cellular $\mathbb{A}^1$-chain complex $\tilde{C}_*^{cell}(\mathbb{P}^n)$ can be expressed in the following form:

\begin{equation*}
    \KMW_n \xrightarrow{\partial_n} \cdots \xrightarrow{\partial_{i+1}} \KMW_i \xrightarrow{\partial_{i}} \cdots \KMW_1 \xrightarrow{\partial_{1}} \mathbb{Z}
\end{equation*}

where the boundary maps $\partial_i$ are defined as follows:

\begin{equation*}
    \partial_i =
    \begin{cases}
        0, & \text{if $ i $ is odd;} \\
        \eta, & \text{if $ i $ is even.}
    \end{cases}
\end{equation*}

Consequently, the cellular homology groups $ H_i^{cell}(\mathbb{P}^n) $ have the structure:

\begin{equation*}
  \mathbf{H}_i^{cell}(\mathbb{P}^n) =
    \begin{cases}
        \mathbb{Z}, & \text{if } i = 0; \\
        \KMW_i / \eta, & \text{if } i < n \text{ and is odd; } \\
        _{\eta} \KMW_i, & \text{if } i < n \text{ and is even.}
    \end{cases}
\end{equation*}

Furthermore, we have the following identities for the top-dimensional homology:

\begin{equation*}
    \mathbf{ H }_n^{cell}(\mathbb{P}^n) =
    \begin{cases}
       \KMW_n, & \text{if } n \text{ is odd;} \\
        _{\eta} \KMW_n, & \text{if } n \text{ is even.}
    \end{cases}
\end{equation*}
\end{corollary}

By applying Proposition \ref{cellCohomology} with $ M = \KM_{n} $ and $ \KMW_{n} $, we can reproduce the following result found in \cite{fasel2013projective}:

\begin{corollary}
    The Chow group of $\mathbb{P}^n$ is given by:

\begin{equation*}
    \mathrm{ CH }^i(\mathbb{P}^n) = \mathbb{Z}.
\end{equation*}

For $0 \leq i \leq n$, the Chow-Witt group of $\mathbb{P}^n$ is given as follows:

\begin{equation*}
  \widetilde{\mathrm{CH}}^i(\mathbb{P}^n) =
    \begin{cases}
        GW(k), & \text{if } i=0 \text{ or } i=n \text{ with } n \text{ odd;} \\
        \Z, & \text{if } i \text{ is even and } i \neq 0; \\
        2\Z, & \text{if } i \text{ is odd and } i \neq 0.
    \end{cases}
\end{equation*}
\end{corollary}

Notice that projective spaces are toric, this example motivates us to extend our computations to smooth toric varieties.

\subsection{Cubical Cells}
To enhance the clarity of our computations, we now define explicit cells. We will assume that $Y_j \cong \Gm^{n_j} = \{(x_1, \ldots, x_{n_j}) \mid x_i \neq 0\}$. Denote $\mathbf{H} := \bfZa[\Gm] \cong \KMW_1 \oplus \mathbb{Z}$; hence, $\bfZa[Y_j] \cong \mathbf{H}^{\otimes n}$.

Next, we define a cubical cell $e$ as a subvariety in $Y_j$ such that

\[
e = e_1 \times \cdots \times e_{n_j}, \quad \text{where } \Gm \supset e_i = \{x_i = 1\} \text{ or } \{x_i \neq 0\}.
\]

We further define the following sets:

\[
\tau_e^1 = \{i \mid e_i = \{x_i = 1\}\}, \quad \tauO_e = \{i \mid e_i = \{x_i \neq 0\}\}, \quad \tau_e = \tau_e^1 \sqcup \tauO_e.
\]

Let the twist degree of $e$ be denoted as $t_e := |\tauO_e|$. It follows that $e \cong \Gm^{t_e}$ and the cubical cell $e$ can be interpreted as an embedding $\iota_e: \Gm^{t_e} \hookrightarrow Y_j$. This cubical cell induces a splitting injection (i.e., a direct summand) $[e]$ of $\bfZa[Y_j]$ given by:

\[
[e]: \KMW_{t_e} \to \bfZa[\Gm^{t_e}] \xr{\iota} \bfZa[Y_j].
\]

This construction provides a decomposition of $\bfZa[Y_j]$ in terms of cubical cells:

\[
\bfZa[Y_j] \cong \bigoplus_{e \subset Y_j \text{ cubical}} [e](\KMW_{t_e}).
\]

Finally, we define the oriented cubical cell. Let $\theta: \mathbb{A}^n/\afnz{n} \times Y_j \to Th(\nu_n|_{Y_j})$ be an orientation of the Thom space of the normal bundle $Th(\nu_n)$. This orientation induces an oriented cell.

\[
[e,\theta]: \KMW_{n} \otimes \KMW_{t_e} \to \wt{\mathbf{H}}^{\af^1}_n(Th(\nu_n))
\]
This map is induced by the framed embedding $ \wt{\iota_{e,\theta}}: \af^n/ \afnz{n} \times \Gm^{t_e} \hookrightarrow Th(\nu_n|_{Y_j}) $. A similar decomposition can be expressed as follows:

\[
C^{cell}_n(X)=\wt{\mathbf{H}}^{\af^1}_n(Th(\nu_n)) \cong \bigoplus_{e \subset Y_{nj} \text{ cubical}} [e,\theta_n](\KMW_n \otimes \KMW_{t_e})
\]

When the orientation is clear, we may sometimes omit $\theta$, denoting $ [e] = [e,\theta] $, and let $ |e| = n $ signify the dimension of the cell. We can interpret $ [e] $ as an analogue of a generator.

If $ |e| + t_e > 0 $, for any integer $ i \in \bZ $ satisfying $ i \leq |e| + t_e $, there exists a natural action of $ \KMW_i $ on $ [e] $. Specifically, for $ \kappa \in \KMW_i $:

\[
\kappa[e]: \KMW_{|e|+t_e-i} \to C^{cell}_{|e|}(X), \quad \lambda \mapsto [e](\kappa\lambda)
\]

\begin{example}[Cubical (Koszul) Cellular Structure]
\label{cubicalAn}
We can endow the affine space $ \af^n $ with a canonically oriented cellular structure. We define:

\[
\Omega_i(\af^n) := \{ (x_1, \ldots, x_n)\in \af^n \mid \exists J\subset \llBra{1,n}, |J|=n-i, \forall j\in J, x_j \neq 0 \}
\]

Thus, the closed complement is given by:

\[
\Omega_i \setminus \Omega_{i-1} = \{ (x_1, \ldots, x_n)\in \af^n \mid \exists J\subset \llBra{1,n}, |J|=n-i, \forall j\in J, x_j \neq 0, \forall k \in \llBra{1,n}\setminus J, x_k = 0 \}
\]

It is evident that:

\[
\Omega_i \setminus \Omega_{i-1} = \bigsqcup_{J_i\subset \llBra{1,n}, |J|=n-i} Y_{J_i}
\]

where each component $ Y_{J_i} $ is isomorphic to $ \Gm^{n-i} $. Additionally, for each $ Y_{J_i} $, we have the canonical orientation of the normal bundle given by $ \aBra{\ldots,x_k,\ldots} $ where $ k \in \llBra{1,n}\setminus J $.

Consequently, we establish the cellular chain complex:

\[
\wt{C}^{cell}_i (\af^n) = \bigoplus_{J_i\subset \llBra{1,n}, |J|=n-i} \KMW_i \otimes \bfZa[Y_{J_i}]
\]

The oriented cubical cells of $ \af^n $ are represented as:

\[
[e]=[e_1]\otimes \cdots \otimes [e_n], \quad [e_i]= [\{x_i=1\}], [\{x_i\neq 0\}] \text{ or } \aBra{x_i}
\]

Here, $ \aBra{x_i}: \af^1 /\Gm \to Th(\nu_i) $ represents the orientation induced by $ x_i $. 

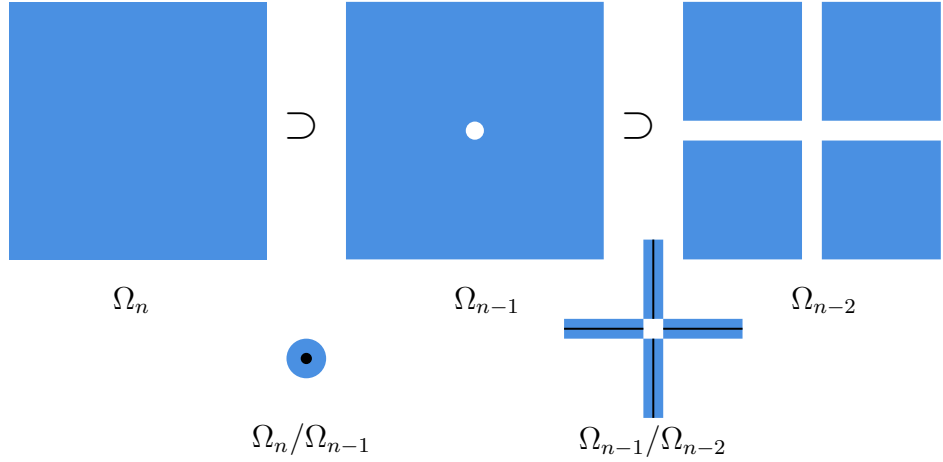
\begin{center}

\tikzset{every picture/.style={line width=0.75pt}} 

\begin{tikzpicture}[x=0.75pt,y=0.75pt,yscale=-1,xscale=1]

\draw  [draw=none][fill={rgb, 255:red, 74; green, 144; blue, 226 }  ,fill opacity=0.5 ] (10,20) -- (140,20) -- (140,150) -- (10,150) -- cycle ;
\draw  [draw=none][fill={rgb, 255:red, 74; green, 144; blue, 226 }  ,fill opacity=0.5 ] (310,20) -- (310,150) -- (180,150) -- (180,20) -- (310,20) -- cycle (240.42,85) .. controls (240.42,87.53) and (242.47,89.58) .. (245,89.58) .. controls (247.53,89.58) and (249.58,87.53) .. (249.58,85) .. controls (249.58,82.47) and (247.53,80.42) .. (245,80.42) .. controls (242.47,80.42) and (240.42,82.47) .. (240.42,85) -- cycle ;
\draw  [draw=none][fill={rgb, 255:red, 74; green, 144; blue, 226 }  ,fill opacity=0.5 ] (480,20) -- (480,80) -- (420,80) -- (420,20) -- (480,20) -- cycle (350,80) -- (350,20) -- (410,20) -- (410,80) -- (350,80) -- cycle (480,150) -- (420,150) -- (420,90) -- (480,90) -- (480,150) -- cycle (350,150) -- (350,90) -- (410,90) -- (410,150) -- (350,150) -- cycle ;
\draw  [draw=none][fill={rgb, 255:red, 74; green, 144; blue, 226 }  ,fill opacity=0.5 ] (150,200) .. controls (150,194.48) and (154.48,190) .. (160,190) .. controls (165.52,190) and (170,194.48) .. (170,200) .. controls (170,205.52) and (165.52,210) .. (160,210) .. controls (154.48,210) and (150,205.52) .. (150,200) -- cycle ;
\draw  [fill={rgb, 255:red, 0; green, 0; blue, 0 }  ,fill opacity=1 ] (162.33,200) .. controls (162.33,198.71) and (161.29,197.67) .. (160,197.67) .. controls (158.71,197.67) and (157.67,198.71) .. (157.67,200) .. controls (157.67,201.29) and (158.71,202.33) .. (160,202.33) .. controls (161.29,202.33) and (162.33,201.29) .. (162.33,200) -- cycle ;

\draw  [draw=none][fill={rgb, 255:red, 74; green, 144; blue, 226 }  ,fill opacity=0.5 ] (340,140) -- (340,180) -- (380,180) -- (380,190) -- (340,190) -- (340,230) -- (330,230) -- (330,190) -- (290,190) -- (290,180) -- (330,180) -- (330,140) -- (340,140) -- cycle (330,180) -- (330,190) -- (340,190) -- (340,180) -- (330,180) -- cycle ;
\draw    (290,185) -- (330,185) ;
\draw    (340,185) -- (380,185) ;

\draw    (335,140) -- (335,180) ;
\draw    (335,190) -- (335,230) ;

\draw (61,162.4) node [anchor=north west][inner sep=0.75pt]    {$\Omega _{n}$};
\draw (147,74.4) node [anchor=north west][inner sep=0.75pt]  [font=\Large]  {$\supset $};
\draw (317,74.4) node [anchor=north west][inner sep=0.75pt]  [font=\Large]  {$\supset $};
\draw (233,162.4) node [anchor=north west][inner sep=0.75pt]    {$\Omega _{n-1}$};
\draw (403,162.4) node [anchor=north west][inner sep=0.75pt]    {$\Omega _{n-2}$};
\draw (131,231.4) node [anchor=north west][inner sep=0.75pt]    {$\Omega _{n} /\Omega _{n-1}$};
\draw (296,232.4) node [anchor=north west][inner sep=0.75pt]    {$\Omega _{n-1} /\Omega _{n-2}$};

\end{tikzpicture}

\end{center}

In a similar vein, we define the sets as follows: 
\[
  \sigma_e=\{i \mid [e_i]=\aBra{x_i}\}, \quad \tau^1_e = \{i \mid e_i = \{x_i=1\}\}, \quad \tauO_e= \{i \mid e_i = \{x_i\neq 0\}\}, \quad \tau_e= \tau^1_e \sqcup \tauO_e
\]

Let the dimension of $[e]$ be denoted by $ |e|=| \sigma_e | $, and the twist degree of $[e]$ be defined as $ t_e:=|\tauO_e | $. It is evident that $ [e] $ is associated with an embedding $ \iota_e : \Gm^{t_e} \hookrightarrow \af^n $. We will denote the image of this embedding as $ e_{\tau} $.

Furthermore, $ [e] $ provides a direct summand given by
\[
  [e]: \KMW_{|e|+t_e} \to C^{cell}_{|e|}(\af^n)
\]
which is induced by the framed embedding $ \wt{\iota_e} $. Additionally, we have the following decomposition:
\[
  C^{cell}_{i}(\af^n) \cong \wt{C}^{cell}_i (\af^n) = \bigoplus_{[e] \text{ oriented cubical, } |e|=i } [e](\KMW_{i+t_e} ) 
\]

Under this orientation, we are able to express the differentials $ \wt{\partial} $ with respect to $ [e] $:
\begin{align*}
  & \wt{\partial}[e]= \sum_{j\in \sigma_e} \aBra{-1}^{(\sigma_e,j)} \epsilon^{(\tauO_e,j)} \partial_j[e] \\
  & = \sum_{j\in \sigma_e} \aBra{-1}^{(\sigma_e,j)} \epsilon^{(\tauO_e,j)}[e_1] \otimes \cdots \otimes [e_{j-1}] \otimes [\{x_j \neq 0\}] \otimes [e_{j+1}] \otimes \cdots \otimes [e_n],
\end{align*}

where $\aBra{-1} = 1 + \eta[-1] \in \KMW_0$, $\epsilon = -\aBra{-1}$, and for a set $S_e$,
\[
  (S_e,j) := | \{ k \in S_e \mid k < j \} |.
\]
 
\end{example}

Through simple computation, and for reasons analogous to those observed in classical cubical complexes, we find that the complex $\wt{C}^{cell}_*(\af^n) \cong \mathbb{Z}$. However, this complex possesses a richer structure.

Even though the cellular homology of $\af^n$ does not yield any new information (as we already know that it is contractible), this cellular structure serves as a fundamental building block for the cellular structure of toric varieties. The term ``cubical" is used because it analogously corresponds to the cellular structure of $\mathbb{D}^n$, which is defined by $n$-dimensional cubes.

From this point onward, we will not differentiate between $C^{cell}_{*}(X)$ and $\wt{C}^{cell}_{*}(X)$.

\begin{example}[Moment-Angle Complex]
  \label{moAng}
  Let $K$ be a simplicial complex on $\llBra{m} = \{1, \ldots, m\}$. The subspace $\AZ_{K} \subset \af^{m}$ is defined as the polyhedral product $(\af^1, \Gm)^{K}$, expressed as follows:
\[
  (\af^1, \Gm)^{K}:= \bigcup_{\sigma \in K} \{ (x_1, \ldots, x_m) \in \af^m \mid x_i \in \Gm \text{ if } i \notin \sigma \}
\]
This space inherits a natural cubical cellular structure from $\af^m$. More specifically, an oriented cubical cell $[e]$ of $\af^m$ belongs to $\AZ_{K}$ if $\sigma_e \in K$. Consequently, the cellular $\af^1$-chain complex is described by the following equation:

\[
  C^{cell}_{i}(\AZ_K) = \bigoplus_{[e] \text{ oriented cubical, } \sigma_e \in K, |e| = i} [e](\KMW_{i+t_e}) 
\]
with differentials identical to those of $\af^m$. The cellular $\af^1$-homology of moment-angle complex is computed in \cite[Theorem C]{hornslien2024}.
\end{example}

\subsection{Toric Action on Cubical Cells}

We will now continue our exploration of the cubical cellular structure of $\af^n$ as detailed in Example \ref{cubicalAn}.

There exists a natural toric action $ \Gm^n \times \af^n \to \af^n $ on $\af^n$. For an oriented cubical cell $[e]$, we consider a group section $ g: \Gm^{t_e} \to \Gm^n $, which can be represented by a $ t_e \times n \ \mathbb{Z}$-matrix $\{r_{ij} = \log g_{ij}\}_{i \in \tau^{\odot}_e, j \in \llBra{n}}$. Let $M_{ij}$ denote the basis of the group sections, corresponding to the matrix $ E_{ij} $. In this context, $g$ acts on the framed morphism $\wt{\iota_e}$, thereby inducing $g_*[e]$. 

We note that $g = g_{\tau} \times g_{\sigma} $, where $g_{\sigma}: \Gm^{t_e} \to \Gm^{|e|} $ acts non-trivially only on the coordinates within $\sigma_e$, among others. For the sake of brevity, we will use the notation $ g_{\sigma} = g_{\sigma} \times \mathrm{id} $ when referring to the action. Consequently, we have $ g_*[e] = g_{\tau*} g_{\sigma*}[e] $.

For a subset $\omega \subset \tau_e$, the notation $[e_{\omega}]$ denotes the cell that shares the same $\sigma_e = \sigma_{e_\omega}$ but satisfies $ \tauO_{e_\omega} = \omega $. We will first focus on the action of $g_{\sigma}$. The following lemma generalizes \cite[Lemma 2.48]{MOREL2023109346}:

\begin{lemma}
\label{tActSigma}
Let $ g_{\sigma}: \Gm^{t_e} \to \Gm^{|e|} $ be a group section. Then we have  
\[
g_{\sigma*}[e] = \sum_{\omega \subset \tau^\odot_e} \aBra{-1}^{\sum_{i \in \omega} \sum_{j \in \sigma_e} r_{ij}} \left( \prod_{i \in \tau^\odot_e \setminus \omega} \chi\left( \sum_{j \in \sigma_e} r_{ij} \right) \right) \eta^{t_e - |\omega|} [e_{\omega}]
\]
where the character function $ \chi: \mathbb{Z} \to \{0,1\} $ is defined by $ n \mapsto i $ for $ i \equiv n \mod 2 $.

\end{lemma}
\begin{proof}
First, we observe that $ g_{\sigma} = \prod_{i \in \tau^{\odot}_i} g_{i\sigma} = \prod_{i \in \tau^{\odot}_e} \prod_{j \in \sigma_e} g_{ij} = \prod_{i \in \tau^{\odot}_e} \prod_{j \in \sigma_e} M_{ij}^{r_{ij}} $. Therefore, it suffices to understand the action of $ M_{ij} $. 

We can interpret $ M_{ij} $ as being induced by the mapping $ (\ldots, x_i, \ldots, x_j, \ldots) \mapsto (\ldots, x_i, \ldots, x_ix_j, \ldots) $. Utilizing a similar argument as in \cite[Lemma 2.48]{MOREL2023109346}, we find:

\begin{align*}
M_{ij*} &: \cdots[\lambda_i]\cdots[\lambda_j]\cdots \mapsto \cdots\left([\lambda_i] \otimes \cdots \otimes 1 + 1 \otimes \cdots \otimes [\lambda_i] + [-1] \otimes \cdots \otimes [\lambda_i]\right)\cdots[\lambda_j]\cdots \\ 
&\mapsto \cdots[\lambda_i]\cdots[\lambda_j]\cdots + \cdots \widehat{[\lambda_i]}\cdots\eta[\lambda_i][\lambda_j]\cdots + \cdots[-1]\cdots\eta[\lambda_i][\lambda_j]\cdots \\ 
&= (1 + [-1]\eta)\cdots[\lambda_i]\cdots[\lambda_j]\cdots + \eta\cdots[\lambda_i]\cdots[\lambda_j]\cdots
\end{align*}

Thus, we have 
\[
M_{ij*}[e](\cdots[\lambda_i]\cdots[\lambda_j]\cdots) = \left(\aBra{-1}[e] + \eta[e|_{\hat{i}}]\right)(\cdots[\lambda_i]\cdots[\lambda_j]\cdots)
\]
where the cell $ [e|_{\hat{i}}] = \cdots \otimes [\{ x_i=1 \}] \otimes \cdots $, which represents the transference of $ i $ from $ \tau^\odot $ to $ \tau^1 $.

Since $ i \notin \tau^\odot_{e|_{\hat{i}}} $, $ M_{ij} $ acts trivially on $ [e|_{\hat{i}}] $. Consequently, we can conclude 
\[
g_{ij*}[e] = \aBra{-1}^{r_{ij}}[e] + \chi(r_{ij})\eta[e|_{\hat{i}}].
\]
It then follows that
\[
g_{i\sigma*}[e] = \aBra{-1}^{\sum_{j \in \sigma_e} r_{ij}}[e] + \chi\left(\sum_{j \in \sigma_e} r_{ij}\right)\eta[e|_{\hat{i}}].
\]
Finally, we arrive at
\begin{align*}
g_{\sigma*}[e] &= \sum_{\omega \subset \tau^\odot_e} \aBra{-1}^{\sum_{i \in \tau^\odot_e \setminus \omega} \sum_{j \in \sigma_e} r_{ij}} \left( \prod_{i \in \omega} \chi\left(\sum_{j \in \sigma_e} r_{ij}\right) \right) \eta^{|\omega|} [e|_{\hat{\omega}}].
\end{align*}

\begin{align*}
  &=\sum_{\omega \subset \tau^\odot_e} \aBra{-1}^{\sum_{i \in \omega} \sum_{j \in \sigma_e} r_{ij}} \left( \prod_{i \in \tau^\odot_e \setminus \omega} \chi\left(\sum_{j \in \sigma_e} r_{ij}\right) \right) \eta^{t_e - |\omega|} [e_{\omega}]
\end{align*}
\end{proof}

The action of $ g_\tau $ induces a new morphism, denoted $ g_\tau \cdot e : \Gm^{t_e} \to \Gm^{|\tau_e|} $. This morphism indeed defines a cell (although it may not be cubical), represented as $ [g_\tau \cdot e]: \KMW_{t_e} \to \mathbf{Z}_{\mathfrak{a}^1}[\Gm^{|\tau_e|}] $. By definition, we know that $ g_{\tau*}[e] = [g_\tau \cdot e] $. By combining this with the previous lemma, we obtain the following results:
\[
g_{*}[e] = \sum_{\omega \subset \tau^\odot_e} \aBra{-1}^{\sum_{i \in \omega} \sum_{j \in \sigma_e} r_{ij}} \left( \prod_{i \in \tau^\odot_e \setminus \omega} \chi\left(\sum_{j \in \sigma_e} r_{ij}\right) \right) \eta^{t_e - |\omega|} [g_\tau \cdot e_{\omega}]
\]

Formulating expressions in terms of cubical cells is generally more complex. Let $ h = 1 + \aBra{-1} $ and $ r'_{ij} = r_{ij} + \delta_{ij} $.

\begin{proposition}
  \label{tActMain}
  Let $ g: \Gm^{t_e} \to \Gm^{n} $ be a group section. Then, the induced action by $ g $ is given by:

  \[
  g_*[e] = \sum_{\omega \subset \tau_e} c(g,e)_\omega [e_{\omega}].
  \]

  For subsets $ \omega \subset \tau_e $ such that $ t_e = |\omega| $, we have:

  \[
  c(g,e)_\omega = \frac{1}{2} \left( \mathrm{det}([r'_{ij}]_{i \in \tauO_e}^{j \in \omega}) h - \eta[-1] \sum_{\pi \subset \omega} (-1)^{|\omega| - |\pi|} \prod_{i \in \tauO_e} \chi\left(\sum_{j \in \omega \sqcup \sigma_e} r'_{ij}\right) \right).
  \]

  Furthermore, if $ t_e > |\omega| $, then:

  \[
  c(g,e)_\omega = \eta^{t_e - |\omega|} \sum_{\pi \subset \omega} (-1)^{|\omega| - |\pi|} \prod_{i \in \tauO_e} \chi\left(\sum_{j \in \omega \sqcup \sigma_e} r'_{ij}\right).
  \]

  Conversely, if $ t_e < |\omega| $, then:

  \[
  c(g,e)_\omega = [-1]^{|\omega| - t_e} (-2)^{t_e - |\omega|} \sum_{\pi \subset \omega} (-1)^{|\omega| - |\pi|} \prod_{i \in \tauO_e} \chi\left(\sum_{j \in \omega \sqcup \sigma_e} r'_{ij}\right).
  \]
\end{proposition}

The proof of this proposition will be presented in \nameref{sec:appendix}. The results can be further simplified if additional information regarding $ g $ or $ r_{ij} $ is available.

\begin{corollary}
  \label{tActSupp}
  Let $ \omega_0 \subset \tau_e $ be the subset such that $ j \in \omega_0 $ if $ \forall i \in \tauO_e, r'_{ij} = 0 $. Then we can express $ g_*[e] $ as follows:

  \[
  g_*[e] = \sum_{\omega \subset \tau_e \setminus \omega_0} c(g,e)_\omega [e_{\omega}].
  \]

  Furthermore, if $ t_e > |\tau_e| - |\omega_0| $, then:

\[
g_*[e] = \eta^{t_e - |\tau_e| + |\omega_0|} \sum_{\omega \subset \tau_e \setminus \omega_0} \eta^{|\tau_e| - |\omega_0| - |\omega|} \sum_{\pi \subset \omega} (-1)^{|\omega| - |\pi|} \prod_{i \in \tauO_e} \chi\left(\sum_{j \in \pi \sqcup \sigma_e} r'_{ij}\right) [e_{\omega}]
\]
\end{corollary}
\begin{proof}
  It is evident from the formula presented in Proposition \ref{tActCompute} that $ c(g,e)_{\omega} = 0 $ if $ \omega \cap \omega_0 \neq \emptyset $.
\end{proof}

\section{Cellular Structures of Smooth Toric Varieties}

We begin by defining an $ m $-edge, $ n $-dimensional fan $ \Sigma = (K, \lambda) $ based on the following data: 
\begin{enumerate}
  \item $ K $ is a simplicial complex on $ \llBra{m} $, where we denote $ c \in K $ as a $ |c| $-dimensional cone;
  \item $ \lambda: \mathbb{Z}^m \to \mathbb{Z}^n $ is a group morphism, referred to as the characteristic function, allowing us to define the induced morphism $ \exp(\lambda): \mathbb{G}_m^m \to \mathbb{G}_m^n $.
\end{enumerate}
For our purposes, we require the fan to satisfy a non-degeneracy condition: for any $ \sigma \in K $, the vectors $ \{ \lambda(v_i) \}_{i \in \sigma} $ must be linearly independent in $ \mathbb{Z}^n $ and form a basis for the lattice generated by the cone. We denote $ K_{\max} $ (the facets) as the subset of maximal elements of $ K $, and we let $ K_i $ represent the subset of $ c \in K $ such that $ |c| = i $.

For our objectives, we conceptualize a smooth toric variety $ X_{\Sigma} $ associated with the fan $ \Sigma $ as a variety that admits a ``good" covering $ \bigcup_{c \in K} U_c $, where we have isomorphisms $ \phi_c: U_c \xrightarrow{\cong} \mathbb{G}_m^{n - |c|} \times \mathbb{A}^{|c|} $. This covering possesses the following properties: if $ a \subset b $, then $ U_a \subset U_b $, and $ U_{a} \cap U_{b} = U_{a \cap b} $. Moreover, if $ c \subset \sigma_1 $ and $ c \subset \sigma_2 $ for $ \sigma_i \in K_{\max} $, then the embeddings $ \phi_{\sigma_i} \phi^{-1}_{c}: \mathbb{G}_m^{n - |c|} \times \mathbb{A}^{|c|} \to \phi_{\sigma_i}(U_{\sigma_i}) \hookrightarrow \mathbb{A}^n $ differ by a group section $ g_{\sigma_1 \sigma_2}: \mathbb{G}_m^{n - |c|} \to \mathbb{G}_m^n $, which can be computed from $ \lambda $.

Thus, we establish an evident cellular structure for $ X_{\Sigma} $:

\[
\Omega_i(X_{\Sigma}) := \bigcup_{\sigma \in K_i} U_{\sigma}
\]
The closed complement is represented as 
\[
\Omega_i \setminus \Omega_{i-1} = \bigsqcup_{\sigma \in K_i} Z_{\sigma} \text{, where } Z_{\sigma} := U_{\sigma} \setminus \bigcup_{\tau \in K, \tau \subset \sigma} U_{\tau}
\]
Each $ i $-dimensional cone $ \sigma \in K_i $ also generates an embedding $ \iota_{\sigma}: \mathbb{G}_m^{n-i} \times \mathbb{A}^{i} \xrightarrow{\cong} U_{\sigma} \hookrightarrow X_{\Sigma} $. In fact, the embedding $ \iota_{\sigma} $ induces another embedding $ \iota_{\sigma}: \mathbb{G}_m^{n-i} \xrightarrow{\cong} Z_{\sigma} $. Consequently, for any cubical cell $ [e_\omega] $ in $ \mathbb{G}_m^{n-i} $, it results in an unoriented cell $ [e^{\sigma}_\omega] $. It is noteworthy that if we restrict these cells to a maximal affine toric subvariety $ U_{\sigma_{\max}} \hookrightarrow \mathbb{A}^n $, they coincide with the cubical cells illustrated in Example \ref{cubicalAn}.

The $i$-th term of the cellular $\af^1$-chain complex for the space $X_\Sigma$ is defined as follows:

\[
  C^{cell}_i(X_\Sigma) \cong \bigoplus_{\sigma \in K_i} \bigoplus_{e^\sigma_\omega \subset Z_{\sigma} \text{ cubical}} [e^{\sigma}_\omega](\KMW_{i+|\omega|}).
\]

The assignment of orientations and the definition of differentials in this context are more complex. To begin, let us impose a total order on $K_{max}$. For each $\sigma \in K$, we define $f(\sigma) \in K_{max}$ as the first maximal cone such that $\sigma \subset f(\sigma)$. The orientation of $ [e^\sigma] $ is then determined by the morphism $\phi_{f(\sigma)} \phi^{-1}_{\sigma}: \Gm^{n-|\sigma|} \times \af^{|\sigma|} \to \phi_{f(\sigma)}(U_{f(\sigma)}) \hookrightarrow \af^n$. However, when a cone is shared among several maximal cones, each cone induces a distinct orientation. Consequently, it becomes necessary to establish a method for comparing different orientations to compute the differentials in the cellular $\af^1$-chain complex. It is important to note that changes in orientation are consistently induced by the action of a group section $g: \Gm^{n-|\sigma|} \to \Gm^{n}$. We can thus apply Proposition \ref{tActMain} to obtain the following result:

\begin{proposition}
  For a smooth toric variety $X_\Sigma$, the cellular $\af^1$-chain complex is defined as follows:

\[
  \wt{C}^{cell}_*(X_\Sigma) = \left( \bigoplus_{\sigma \in K_i} \bigoplus_{e^\sigma_\omega \subset Z_{\sigma} \text{ cubical}} [e^{\sigma}_\omega](\KMW_{i+|\omega|}), \partial \right)_i,
\]

where the differentials $\partial$ take the following form:

\[
  \partial[e^\sigma_\omega] = \sum_{j \in \sigma} \sum_{\substack{\pi \subset \{j\} \sqcup f(\sigma) \setminus \sigma \\ |\pi| = |\omega| + 1}} c_{j,\pi} [e^{\sigma \setminus \{j\}}_{\pi}] + [-1](\ldots) + \eta(\ldots).
\]

Here, $c_{j,\pi} = \mathrm{det}([r'_{ij}(g)]^{j \in \pi}_{i \in \omega \sqcup \{j\}})$ represents the change of orientation induced by the morphism $g: \Gm^{n-|\sigma|+1} \to \Gm^{n}$.
\end{proposition}

\begin{proof}
  This result follows from the combination of Example \ref{cubicalAn} and Proposition \ref{tActMain}.
\end{proof} 

\begin{corollary}\label{rationDecomp}
  In $\DM(k)$ (resp. $\DMt(k)$), the motive $\M(X_{\Sigma})$ (resp. $\Mt(X_{\Sigma})$) is Tate, and the extensions consist exclusively of $l \in \mathbb{Z}$, $l$-extensions, and $[-1]^{p}$-extensions (resp. as well as $\eta^{q}$-extensions). In particular, if we consider the rational motive in $\DM(k)_{\mathbb{Q}}$ where $l$ is invertible and $\eta = [-1] = 0$, then $\M(X_\Sigma)_{\mathbb{Q}}$ is expressed as a direct sum of terms of the form $\mathbb{Q}(q)[p]$.
\end{corollary}

To compute the change of orientation $g$ and the cellular $\af^1$-homology explicitly, we must introduce the concept of the moment-angle complex.

\subsection{Moment-Angle Complex and Toric Quotient}

As noted in Example \ref{moAng}, we defined $\AZ_K$. There exists an induced morphism given by the homogeneous coordinate $p: \AZ_K \to X_\Sigma$. In fact, it induces an isomorphism, as stated in \cite[Theorem 2.1]{cox1995homogeneous}:

\[
  \mathrm{CoKer} \exp(\lambda) \times \left( \AZ_K / \Ker \exp(\lambda) \right) \cong X_\Sigma.
\]

We will focus on the cases where $\lambda$ is surjective and $m > n$. For the remainder of this discussion, we assume that the space $ K $ has pure dimension $ n $, that is, $ K_{\max} = K_{n} $.

The projection also induces a morphism of complexes in $ \Daba{k} $:
\[
  p_*: C^{cell}_*(\AZ_K) \to C^{cell}_* (X_\Sigma).
\]

For a cell $[e]$ of $\AZ_K$, let $\sigma_e \in K$. This cell is termed \textit{canonical} if $\tauO_e \subset f(\sigma_e$), which is equivalent to $\tau^1_e \supset \llBra{m} \setminus f(\sigma_e)$. We denote the subgroup generated by canonical cells as $ C^{can}_i(\AZ_K) \subset C^{cell}_i(\AZ_K) $.

For any cell $[e]$, we define the canonical transformation as a group section $ T_e: \Gm^{t_e} \to \Ker \exp(\lambda) $ with the following properties:
\[
  \forall j \in \tau^1_e \setminus f(\sigma_e), \forall i \in \tauO_e, \ r_{ij}(T_e) = 0 ; \quad \forall j \in \tauO_e \setminus f(\sigma_e), \forall i \in \tauO_e, \ r_{ij}(T_e) = -\delta_{ij}.
\]

\begin{lemma}
  For any cell $[e]$ of $[\AZ_K]$, $ T_e $ is uniquely defined and $ T_{e*}[e] \in C^{can}_{|e|}(\AZ_K) $.
\end{lemma}

\begin{proof}
  Let $ r_{ij} = r_{ij}(T_e) $. To define $ T_e: \Gm^{t_e} \to \Ker \exp(\lambda) $, we need to verify that for any $ i \in \tauO_e $, the vector $ \sum_{j \in \llBra{m}} r_{ij}v_j \in \Ker \lambda $. We can decompose this sum as:
  \[
  \sum_{j \in \llBra{m}} r_{ij} v_j = \sum_{j \in \llBra{m} \setminus f(\sigma_e)} r_{ij} v_j + \sum_{j \in f(\sigma_e)} r_{ij} v_j.
  \]
  For $ j \in \llBra{m} \setminus f(\sigma_e) = \tau_e \setminus f(\sigma_e) $, the value of $ r_{ij} $ is determined as $ r_{ij} = -\delta_{ij} $, leading us to define $ u_i = \lambda\left(\sum_{j \in \llBra{m} \setminus f(\sigma_e)} r_{ij} v_j\right) $. 

  For $ j \in f(\sigma_e) $, with $ \lambda(v_j) $ forming a basis, there exists a unique solution ensuring that $ u_i + \sum_{j \in f(\sigma_e)} r_{ij} \lambda(v_j) = 0 $.

  We observe that for all $ j \in \tau_e \setminus f(\sigma_e) $ and all $ i \in \tauO_e $, $ r'_{ij}(T_e) = 0 $. Therefore, according to Corollary \ref{tActSupp}, we conclude that $ \omega_0(T_e) \supset \tau_e \setminus f(\sigma_e) $, which implies that $ T_{e*}[e] = \sum_{\omega \subset f(\sigma_e)} c(T_e, e)_\omega [e_{\omega}] $, where $ [e_{\omega}] $ is canonical by definition.

\end{proof}

It is clear that if $[e]$ is a canonical cell, then $ T_e = \mathrm{id} $. We now define a differential for $ C^{can}_*(\AZ_K) $:
\[
  \partial[e] = \sum_{j \in \sigma_e} (-1)^{(\sigma_e,j)} \epsilon^{(\tauO_e,j)} T_{\partial_j[e]*} \partial_j[e].
\]

Let $ T_*[e] = T_{e*}[e] $. This results in an induced morphism of complexes:
\[
  T_*: C^{cell}_*(\AZ_K) \to C^{can}_*(\AZ_K).
\]
Additionally, $ p_* $ induces a morphism of complexes:
\[
  p_{can}: C^{can}_*(\AZ_K) \to C^{cell}_* (X_\Sigma).
\]

\begin{proposition}\label{canIso}
  The map $ p_{can} $ is an isomorphism of complexes; furthermore, we have $ p_{can} \circ T_* = p_* $.  
\end{proposition}
\begin{proof}
  It suffices to observe that for a canonical element $[e]$, the cycle $ e: \Gm^{t_e} \to \AZ_K $ intersects at most one point with the orbit of $ \Ker \exp(\lambda) $. Consequently, the composition $ p \circ e: \Gm^{t_e} \to \AZ_K \to X_\Sigma $ is an embedding. Therefore, $ p_{can}: C^{can}_i(\AZ_K) \to C^{cell}_i (X_\Sigma) $ is an isomorphism for each $ i $, establishing that it is indeed an isomorphism of complexes.
\end{proof}

Consequently, to understand the space $ X_\Sigma $, we need only to examine the complex $ C^{can}_*(\AZ_K) $. Furthermore, we can simplify $ C^{can}_*(\AZ_K) $ by considering the expansion and restriction of simplicial complexes.

Given the ordered set $ K_{\max} =\{\sigma_1,\ldots,\sigma_l\} $, we define the restriction $ r(\sigma_i) = \bigcup_{ \tau \in \min( \sigma_i )} \tau $ as the union of the sets in $ \min( \sigma_i ) = \{ \tau \subset \sigma_i \mid \tau \text{ is minimal for } f(\tau)=\sigma_i \} $. Recall that $ f(\tau)=\sigma_i  \Leftrightarrow \tau \subset \sigma_i, \forall j<i, \tau \not\subset \sigma_j $.

Next, we define the subcomplex $ \overline{ C }^{can}_*(\AZ_K) \subset C^{can}_*(\AZ_K) $ as $ [e^\sigma_\omega] \in \overline{C}^{can}_*(\AZ_K) $ if $ \omega \sqcup \sigma \subset r(f(\sigma)) $, with the differential given by projection onto the direct summand.

\begin{proposition}\label{restrQis}
The inclusion $ \overline{ C }^{can}_*(\AZ_K) \subset C^{can}_*(\AZ_K) $ is a retract; specifically, this inclusion is a quasi-isomorphism.
\end{proposition}
\begin{proof}
  We will prove this by induction on the length $ l $ of $ K_{\max} = \{\sigma_1,\ldots,\sigma_l\} $. Let $ K_{\leq t} $ be the simplicial complex defined by $ \{\sigma_1,\ldots,\sigma_t\} $.
  
  For $ t=1 $, we have $ r(\sigma_1) = \emptyset $, and it follows that $ C^{can}_*(\AZ_{K_{\leq 1}}) \cong C^{cell}_*(\mathbb{A}^n) $. It is evident that this is quasi-isomorphic to $ C^{can}_*(\AZ_{K_{\leq 1}}) = [e^{\emptyset}_{\emptyset}] \mathbb{Z} $.

  Assuming the statement holds for $ t-1 $, we must demonstrate that the quotient $ C^{can}_*(\AZ_{K_{\leq t}})/C^{can}_*(\AZ_{K_{\leq t-1}}) $ is quasi-isomorphic to $ \overline{C}^{can}_*(\AZ_{K_{\leq t}})/\overline{C}^{can}_*(\AZ_{K_{\leq t-1}}) $. In particular, let $ R_{\sigma_t} = \{ \tau \subset \sigma_t \mid f(\tau) = \sigma_t \} $. Therefore, we have isomorphism $ C^{can}_*(\AZ_{K_{\leq t}})/C^{can}_*(\AZ_{K_{\leq t-1}}) \cong C^{can}_*(\AZ_{R_{\sigma_t}}) $, with the same relationship holding for $ \overline{C}^{can}_* $. Finally, we note that $ C^{can}_*(\AZ_{R_{\sigma_t}}) \cong \overline{C}^{can}_*(\AZ_{R_{\sigma_t}}) \otimes C^{cell}_*(\mathbb{A}^{n-|r(\sigma_t)|}) $.
\end{proof}

Consequently, we can project $C^{cell}_*(\AZ_K)$ to the restriction complex, yielding the following morphism:
\[
  \overline{T}_*: C^{cell}_*(\AZ_K) \to C^{can}_*(\AZ_K) \xr{qis} \overline{C}^{can}_*(\AZ_K).
\]

The set $\{\sigma_1,\ldots,\sigma_l\}$ is referred to as a regular expanding sequence of $K$ if for all $i$, the minimum element $\min(\sigma_i) = \{r(\sigma_i)\}$ contains a unique element. In this case, we have $f(\tau) = \sigma_i$ if and only if $\tau \supset r(\sigma_i)$. In regular cases, it follows that $[e^\sigma_\omega] \in \overline{C}^{can}_*(\AZ_K)$ if and only if $[e^\sigma_\omega] = [e^{r(\sigma_i)}_\emptyset]$ for some facet $\sigma_i$. This condition simplifies computations significantly.

From this point forward, we will concentrate on shellable simplicial complexes that admit a regular expanding sequence.
\subsection{Shellable Toric Variety}
\label{sec:shellable}

We adopt the definition provided in \cite{cai2021integral}. An ordering $K_{\max} = \{\sigma_1,\ldots,\sigma_l\}$ is defined as a shelling if the complex $B_j := (\bigcup_{1\leq i<j} \sigma_i) \cap \sigma_j$ is pure of dimension $n-1$; that is, $B_{j,\max} = B_{j,n-1}$. It is readily apparent that a shelling constitutes a regular expanding sequence as well. If $K$ possesses a shelling, we then refer to $K$ as shellable. From this point onward, we will assume that $K$ is shellable.

The following proposition is analogous to \cite[Prop 3.6]{cai2021integral}. Given a mod-$2$ linear function $\kappa: \mathbb{Z}^n \to \mathbb{Z}_2^n \to \mathbb{Z}_2$, we define the row set $\omega_{\kappa} \subset \llBra{m}$ as the subset $\{j \in \llBra{m} \mid \kappa \lambda(v_j) \equiv 1 \mod 2\}$, where $v_j$ are the basis vectors of $\mathbb{Z}^m$. Let $\mathrm{row} \lambda \subset 2^{\llBra{m}}$ denote the set of all row sets; thus, we can observe that $|\mathrm{row} \lambda| = 2^n$.

\begin{proposition}\label{cellTrans}
  For any cell $[e]$ of $\AZ_K$, we have that $\overline{T}_*[e] \neq 0$ if and only if $\sigma_e = r(\sigma_e) := r(f(\sigma_e))$ and $\tauO_e \cap f(\sigma_e) = \emptyset$. In such cases, we have:
  \[
    \overline{T}_*[e] = \eta^{t_e}[e_{\emptyset}] \text{ or } 0.
  \]
  Furthermore, if for some $\omega_\kappa \in \mathrm{row} \lambda$, the conditions $\sigma_e = \omega_\kappa \cap f(\sigma_e)$ and $\tauO_e \subset \omega_\kappa$ hold, then it follows that:
  \[
    \overline{T}_*[e] = \eta^{t_e}[e_{\emptyset}].
  \]
\end{proposition}

\begin{proof}
  Let $\omega_0 = \omega_0(T_e)$ and $r_{ij} = r_{ij}(T_e)$. First, observe that $\llBra{m} \setminus f(\sigma_e) \subset \omega_0$. By applying Corollary \ref{tActSupp}, we obtain:
  \[
    T_*[e] = \sum_{\omega \subset f(\sigma_e) \setminus \sigma_e} c(T_e,e)_{\omega} [e_\omega].
  \]
  Recall that $[e^\sigma_\omega] \in \overline{C}^{can}_*(\AZ_K)$ if and only if $[e^\sigma_\omega] = [e^{r(\sigma)}_\emptyset]$. This indicates that $\overline{T}_*[e] = 0$ if $\sigma_e \neq r(\sigma_e)$. For the case where we assume $\sigma_e = r(\sigma_e)$, we derive:
  \[
    \overline{T}_*[e] = c(T_e,e)_{\emptyset} [e_\emptyset].
  \]

We now apply Corollary \ref{tActMain}. If $ t_e = 0 $, then $ T_*[e] = [e] $, as $ [e] $ is already canonical. We can assume $ t_e > 0 $; note that $ r_{ij} = r'_{ij} $ for $ j \in \sigma_e $. Thus, we have

\[
\overline{T}_*[e] = \eta^{t_e} \prod_{i \in \tauO_e} \chi\left( \sum_{j \in \sigma_e} r_{ij} \right) [e_\emptyset].
\]

If $ \tauO_e \cap f(\sigma_e) \neq \emptyset $, then by the definition of $ T_e $, for $ i \in \tauO_e \cap f(\sigma_e) $, it follows that $ r_{ij} = 0 $ for all $ j $. This indicates that $ \overline{T}_e[e] = 0 $.

Finally, if $ \sigma_e = \omega_{\kappa} \cap f(\sigma_e) $, considering $ i \in \tauO_e \subset \omega_\kappa $, we have by the definition of $ T_e $:

\[
-\lambda(v_i) + \sum_{j \in f(\sigma_e)} r_{ij} \lambda(v_j) = 0.
\]

Furthermore, by the definition of $ \omega_{\kappa} $, we have 

\[
1 \equiv \kappa \lambda(v_i) = \sum_{j \in f(\sigma_e)} r_{ij} \kappa \lambda(v_j) \equiv \sum_{j \in \sigma_e = \omega_{\kappa} \cap f(\sigma_e)} r_{ij} \mod 2.
\]

This leads to the conclusion that $ \overline{T}_*[e] = \eta^{t_e}[e_{\emptyset}] $.

\end{proof}

\begin{remark}\label{uniIntersect}
By \cite[Lemma 3.5]{cai2021integral}, for each $ \sigma \in K $, there exists a unique $ \omega_\kappa \in \mathrm{row}\lambda $ such that $ \sigma = f(\sigma) \cap \omega_\kappa $. Moreover, by definition, we have $ r(\sigma) = \sigma \cap \omega_\kappa $ if $ \sigma \in K_{\max} $.
\end{remark}

\subsection{The Cellular Chain Complexes of $ K_\omega $}

For a simplicial complex $ K $ over $ \llBra{m} $, we construct an augmented simplicial chain complex $ C_*(K) $: 

\[
C_i(K) = \bigoplus_{\sigma \in K_i} [\sigma] \mathbb{Z},
\]

where $ [\sigma] = [j_1, \ldots, j_i] $ with $ j_1 < \cdots < j_i \in \sigma $ represents the oriented simplex, and the natural differentials are defined as 

\[
\partial[\sigma] = \sum_{j \in \sigma} (-1)^{(\sigma, j)} [\sigma \setminus \{ j \}].
\]

Similarly, we can define the restriction subcomplex $ \overline{C}_*(K) \subset C_*(K) $, which consists of the set $ \{[\sigma] \mid \sigma \subset r(\sigma) \} $ (in regular cases, $ \sigma = r(\sigma) $). It is evident that this inclusion serves as a quasi-isomorphism.

For a regular expanding sequence $ K_{\max} = \{ \sigma_1, \ldots, \sigma_l \} $, the facet $ \sigma_i $ is termed critical if $ \sigma_i = r(\sigma_i) $. Let $ \mathrm{Cri}(K) $ denote the set of critical facets. The homotopy type only alters at critical facets; therefore, we can further simplify by considering the critical complex $ \overline{C}^{cri}_*(K) $ generated by $ \mathrm{Cri}(K) $ (i.e., the facets $ [\sigma] \in \overline{C}_*(K) $ where $ [\sigma] \in K_{\max} $), along with a differential $ \overline{\partial}^{cri} $. There exists a chain morphism $ \rho : \overline{C}^{cri}_*(K) \xrightarrow{\sim} C_*(K) $ which is a quasi-isomorphism \cite[Lemma 4.4]{cai2021integral}.

Let $ \omega \subset \llBra{m} $. We define the restriction simplicial complex $ K_\omega $ over $ \omega $, where $ \sigma \in K_{\omega} $ if $ \sigma \in K $ and $ \sigma \subset \omega $. Subsequently, we can define a morphism of complexes:

\[
\varphi_{\omega}: C_*(K_\omega) \otimes \KMW_{|\omega|} \to C^{cell}_*(\AZ_K), \quad [\sigma] \mapsto \epsilon^{\sum_{j \in \omega \setminus \sigma} (\omega, j)} [e^{\sigma}_{\omega \setminus \sigma}]
\]
It is straightforward to verify that this morphism is compatible with the differentials. In a similar fashion, we can define the morphism for $ \overline{C}_*(K) \subset C_*(K) $:

\[
\overline{\varphi}_{\omega}: \overline{C}_*(K_\omega) \otimes \KMW_{|\omega|} \to C^{cell}_*(\AZ_K)
\]

For $ \omega \in \mathrm{row}_{\lambda} $, the critical complex $ (\overline{C}^{cri}_*(K_\omega), \overline{\partial}) $ is generated by the set $ \{ r(\sigma_i) \mid r(\sigma_i) = \sigma_i \cap \omega, \text{ for some facet } \sigma_i \in K_{\max} \} $ \cite[Lemma 4.6]{cai2021integral}. This also induces a morphism by composing with $ \rho $:

\[
\overline{\varphi}^{cri}_{\omega}: \overline{C}^{cri}_*(K_\omega) \otimes \KMW_{|\omega|} \to C^{cell}_*(\AZ_K)
\]

Next, we define a morphism $ \phi_{i,\omega} $ for each degree $ i = |r(\sigma)| $, which can be formally expressed as $ ``\frac{1}{\eta^{|\omega|-i}} \overline{T}_* \circ \overline{\varphi}^{cri}_{\omega}" $:

\[
\phi_{i,\omega}: \overline{C}^{cri}_i(K_\omega) \otimes \KMW_{i} \to \overline{C}^{can}_i(\AZ_K), \quad [r(\sigma)] \mapsto [e^{r(\sigma)}_{\emptyset}]
\]

\begin{lemma}\label{phiMor}
The morphism $ \phi_{i,\omega} $ indeed defines a morphism of complexes:
\[
\phi_\omega : (\overline{C}^{cri}_i(K_\omega) \otimes \KMW_{i}, \eta \overline{\partial}^{cri}) \to (\overline{C}^{can}_i(\AZ_K), \overline{\partial})
\]
\end{lemma}

\begin{proof}
We only need to demonstrate the compatibility of differentials:
\[
\phi_{i-1,\omega}(\eta \overline{\partial}^{cri}[r(\sigma)]) = \overline{\partial} \phi_{i,\omega}([r(\sigma)])
\]
We define the morphism $ \widetilde{\phi}_\omega = \overline{T}_* \circ \overline{\varphi}^{cri}_{\omega} $:
\[
\widetilde{\phi}_\omega : (\overline{C}^{cri}_i(K_\omega) \otimes \KMW_{|\omega|}, \overline{\partial}^{cri}) \to (\overline{C}^{can}_i(\AZ_K), \overline{\partial})
\]

By Proposition \ref{cellTrans}, we have:
\[
\widetilde{\phi}_{i,\omega}([r(\sigma)]) = \epsilon^{\sum_{j \in \omega \setminus r(\sigma)} (\omega, j)} T_* [e^{r(\sigma)}_{\omega \setminus r(\sigma)}] = \eta^{|\omega|-i} [e^{r(\sigma)}_{\emptyset}] = \eta^{|\omega|-i} \phi_{i,\omega}([r(\sigma)])
\] 

Since  $ \widetilde{\phi}_\omega $ is a morphism of complexes, we have
\[
\eta^{|\omega| - i + 1} \phi_{i - 1, \omega}(\overline{\partial}^{\text{cri}}[r(\sigma)]) = \widetilde{\phi}_{i - 1, \omega}(\overline{\partial}^{\text{cri}}[r(\sigma)]) = \overline{\partial} \widetilde{\phi}_{i, \omega}([r(\sigma)]) = \eta^{|\omega| - i} \overline{\partial} \phi_{i, \omega}([r(\sigma)]).
\]
Finally, note that since $\eta$ has no zero divisors in $\eta \KMW_* \cong W_*$, we obtain that
\[
\phi_{i - 1, \omega}(\eta \overline{\partial}^{\text{cri}}[r(\sigma)]) = \eta \phi_{i - 1, \omega}(\overline{\partial}^{\text{cri}}[r(\sigma)]) = \overline{\partial} \phi_{i, \omega}([r(\sigma)]).
\]
\end{proof}

We now define 
\[
\overline{C}^{\lambda}_i = \bigoplus_{\omega \in \mathrm{row} \lambda} \overline{C}^{\text{cri}}_i(K_\omega) \otimes \KMW_{i}.
\]
With this definition, we can state and prove our main result:

\begin{theorem}\label{mainComplex}
The morphism of complexes 
\[
\phi_\lambda : (\overline{C}^{\lambda}_i, \eta \overline{\partial}^{\text{cri}}) \to (\overline{C}^{\text{can}}_i(\mathbb{A}^1_K), \overline{\partial}) 
\]
is an isomorphism.

Furthermore, this morphism induces a quasi-isomorphism to the cellular $\mathbb{A}^1$-chain complex $ C^{\text{cell}}_*(X_\Sigma) $ of the smooth shellable toric variety $ X_\Sigma $:
\[
p_{\text{can}} \circ \phi_\lambda : (\overline{C}^{\lambda}_i, \eta \overline{\partial}^{\text{cri}}) \to (C^{\text{cell}}_i(X_\Sigma), \partial). 
\]
\end{theorem}

\begin{proof}
Firstly, we establish that $ \phi_\lambda $ is a well-defined chain morphism, which requires verification that it induces an isomorphism at each degree. In fact, we can express $ \overline{C}^{\lambda}_i $ as follows:
\[
\overline{C}^{\lambda}_i = \bigoplus_{\omega \in \mathrm{row} \lambda} \overline{C}^{\text{cri}}_i(K_\omega) \otimes \KMW_{i} = \bigoplus_{\omega \in \mathrm{row} \lambda} \bigoplus_{\substack{|r(\sigma)| = i \\ r(\sigma) = \omega \cap \sigma \\ \sigma \in K_{\max}}} [r(\sigma)] \KMW_{i}.
\]
From Remark \ref{uniIntersect}, we observe that for each $ \sigma \in K_{\max} $, there exists a unique $ \omega \in \mathrm{row} \lambda $ such that $ r(\sigma) = \omega \cap \sigma $. This uniqueness allows us to conclude that
\[
\overline{C}^{\lambda}_i = \bigoplus_{\substack{|r(\sigma)| = i \\ \sigma \in K_{\max}}} [r(\sigma)] \KMW_{i},
\]
thus leading to the morphism defined by
\[
\phi_{i, \lambda} : \overline{C}^{\lambda}_i \to \overline{C}^{\text{can}}_i(\mathbb{A}^1_K), \quad [r(\sigma)] \mapsto [e^{r(\sigma)}_{\emptyset}].
\]
is an isomorphism for each degree.
\end{proof}

\begin{remark}
  It should be noted that $ \bigoplus_{\omega \in \mathrm{row}\lambda} \overline{C}^{cri}_i(K_\omega) $ is entirely combinatorial; thus, we obtain a result consistent with \cite[Theorem 5.2]{cai2021integral}. Moreover, we can readily utilize computations from real toric manifolds by substituting $ 2 $ with $ \eta $.
\end{remark}

Next, we observe that the reduced homology group $ \widetilde{H}_i(\overline{C}^{cri}_*(K_\omega)) $ is equivalent to $ \widetilde{H}_i(C_*(K_\omega)) = \widetilde{H}_{i-1}(|K_\omega|) $. We define $ G^{\lambda}_{i} = \bigoplus_{\omega \in \mathrm{row}\lambda}\widetilde{H}_{i}(|K_\omega|) $. Let $ B(0) \subset K_{\max} $ serve as a basis for the free part of $ G^{\lambda}_{*,free} $. For $ l > 1 $, let $ B(l) \subset K_{\max} $ be a basis for the $ l $-torsion part of $ G^{\lambda}_{*,l-tor} $. Similarly, we define $ B(l)_i $ for $ G^{\lambda}_{i,free/l-tor} $.

We now introduce $ \KMW_i \sslash l \eta := [ \KMW_{i+1} \xr{l_\sigma \eta}  \KMW_{i} ] \in \Daba{k} $ for $ l \in \mathbb{N}^+ $, and we set $ \KMW_i \sslash 0 = \KMW_i $. This leads us to the following decomposition:

\begin{corollary}
  \label{mainA1cellular}
In the category $\Daba{k}$, for a smooth pure shellable toric variety $X_{\Sigma}$, we have a quasi-isomorphism:
\[
C^{cell}_*(X_\Sigma) \cong \bigoplus_{l \in \mathbb{N}} \bigoplus_{\sigma \in B(l)} \KMW_{|r(\sigma)|} \sslash l \eta [|r(\sigma)|],
\]
for a certain subset $B(1) \subset K_{\max}$, which can be derived from the complex $\bigoplus_{\omega \in \mathrm{row}\lambda} \overline{C}^{cri}_i(K_\omega)$.

Specifically, we establish the correspondences between:
\begin{itemize}
  \item the free summands of $G^{\lambda}_{i-1,free}$ and the free summands of $ \KMW_i[i] $.
  \item the torsion summands of $G^{\lambda}_{i-1,tor}$ and the torsion summands $ \KMW_i \sslash l \eta [i] $ for $l > 1$.
\end{itemize}
As a result, the cellular $\mathbb{A}^1$-homology of $X_\Sigma$ is given by
\[
\mathbf{H}^{cell}_i(X_{\Sigma}) = \bigoplus_{l \in \mathbb{N}} \bigoplus_{B(l)_{i-1}} \KMW_i / l \eta \oplus \bigoplus_{l \in \mathbb{N}^+} \bigoplus_{B(l)_{i-2}} (_{l\eta}\KMW_i).
\]
Moreover, we can establish correspondences between:
\begin{itemize}
  \item the free summands of $G^{\lambda}_{i-1,free}$ and the free summands of $ \KMW_i $.
  \item the torsion summands of $G^{\lambda}_{i-1,tor}$ and the torsion summands of $\KMW_i / l \eta $ for $l > 1$. 
  \item the torsion summands of $G^{\lambda}_{i-2,tor}$ and the torsion summands of ${_{\l \eta}\KMW_i}$ for $l > 1$. 
\end{itemize}
\end{corollary}

\begin{remark}\label{Coxeter}
  For any irreducible root system of type $R$ in a finite dimensional Euclidean space $E$, one can construct the fan $\Sigma_R$, whose maximal cones are Weyl chambers of type $R$. It is also a 
fact that $\Sigma_R$ is a regular complete fan for any irreducible root system of type $R$ \cite{procesi}. Hence, each irreducible root system of type 
$R$ defines a smooth compact toric variety $X_R$ which is known as the Coxeter toric variety of type 
$R$. Let $K_R$ denote the underlying simplicial complex of $X_R$ and $\lambda_R$ denotes the corresponding characteristic map. The Coxeter toric variety of type $A$ is also known as the permutohedral toric variety and is isomorphic to a specific De Concini-Procesi wonderful model as well \cite{CP}. 
In the case of permutohedral toric varieties, $K_\omega$ has no torsion for $\omega\in \mathrm{row}\lambda_{A_n}$ \cite[Remark 2.5]{choi2024cohomologyringsrealpermutohedral}, so the group $\mathbf{H}^{cell}_i(X_{A_n})$ contains only $\eta$-torsion elements by Corollary \ref{mainA1cellular}. 
This is also true for Coxeter toric varieties of type $R=G_2, F_4, E_6$ \cite[Remark 3.5]{cho2019geometricrepresentationsfinitegroups}.
\end{remark}

\begin{remark}
For smooth fans that are not pure or not shellable, as discussed in Section \ref{exotic}, the results presented in Corollary \ref{mainA1cellular} do not hold. However, if we focus solely on the Chow group, we can still derive a basis by removing the non-algebraic geometric components (i.e., the summands $ \bZ(q)[p] $ with $ 2q > p $) and annihilating the higher Chow groups. The following proposition can be viewed as a generalization of \cite[\S 5.2, Theorem p.102-104]{fulton1993introduction} for non-complete cases.
\end{remark}

\begin{proposition}\label{chowDecomp}
  For a smooth toric variety $ X_{\Sigma} $, consider an order on $ K_{\max} = \{\sigma_1, \ldots, \sigma_g\} $. Recall the definition of the sets $ \min(\sigma_i) = \{ \tau \subset \sigma_i \mid \tau \text{ is minimal for } \tau \not\subset \sigma_j \text{ for all } j < i\} $. We then have the following decomposition of the Chow group:

\[
  \mathrm{CH}^*(X_\Sigma) \cong \bigoplus_{\sigma \in K_{\max}} \bigoplus_{\tau \in \min(\sigma)} \bZ [e^{\tau}]
\]

The generators are given by $ [e^{\tau}] \in \mathrm{CH}^{|\tau|}(X_\Sigma) $.
\end{proposition}

\begin{proof}
  We need to adjust the proof of Proposition \ref{restrQis} to obtain a quasi-isomorphism $ \overline{C}^{can}_*(\AZ_K)_{ag} \subset C^{can}_*(\AZ_K)_{ag} $, where the subscript $(-)_{ag}$ indicates reduction modulo the direct summands of $ \KMW_{q}[p] $ for $ q > p $. It is important to note that $ \overline{C}^{can}_*(\AZ_K)_{ag} $ is generated by $ [e^{\tau}] $ where $ \tau \in \min(\sigma) $ and $ \sigma \in K_{\max} $. By Proposition \ref{cellTrans}, we conclude that the differential will always be zero when transitioning to $ \DM(k) $. Since for $ 2j < i $, we have $ \mathrm{H}_{\mathrm{M}}^{i,j}(k) \cong \mathrm{Hom}_{\DM(k)}(\mathbb{Z}, \mathbb{Z}(j)[i]) = 0 $, we can deduce that 
  \[
    \mathrm{CH}^i(X_\Sigma) \cong \mathrm{Hom}_{\DM(k)}(\M(X_\Sigma), \mathbb{Z}(i)[2i]) \cong \mathrm{Hom}_{\DM(k)}(\M(X_\Sigma)_{ag}, \mathbb{Z}(i)[2i]).
  \]
\end{proof}

\section{Examples of Toric Surfaces}
For toric surfaces, we can provide a more explicit description. Given a two-dimensional fan $ \Sigma $ with $ l $ edges, let $ v_i $ denote the basis vector of the $ 1 $-dimensional edge, and let $ \sigma_i $ represent the cone spanned by $ v_i $ and $ v_{i+1} $ (if $ \Sigma $ is complete, we let $ v_{l+1} = v_1 $). If the corresponding toric variety $ X_{\Sigma} $ is smooth, then it holds that the volume $ \mathrm{vol}(v_i, v_{i+1}) = 1 $. Therefore, there exists an integer $ a_{i} $ such that $ v_{i-1} + v_{i+1} = a_{i}v_i $.

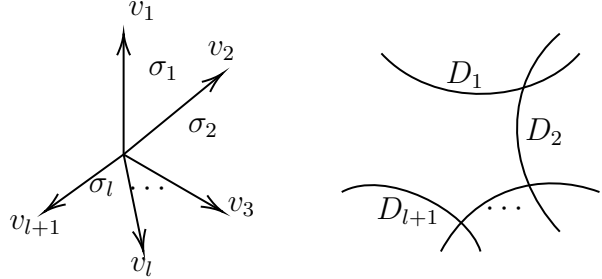
\begin{center}

  \tikzset{every picture/.style={line width=0.75pt}} 
  
  \begin{tikzpicture}[x=0.75pt,y=0.75pt,yscale=-1,xscale=1]
  \draw    (100,101) -- (100,42) ;
  \draw [shift={(100,40)}, rotate = 90] [color={rgb, 255:red, 0; green, 0; blue, 0 }  ][line width=0.75]    (10.93,-3.29) .. controls (6.95,-1.4) and (3.31,-0.3) .. (0,0) .. controls (3.31,0.3) and (6.95,1.4) .. (10.93,3.29)   ;
  \draw    (100,101) -- (148.45,61.27) ;
  \draw [shift={(150,60)}, rotate = 140.65] [color={rgb, 255:red, 0; green, 0; blue, 0 }  ][line width=0.75]    (10.93,-3.29) .. controls (6.95,-1.4) and (3.31,-0.3) .. (0,0) .. controls (3.31,0.3) and (6.95,1.4) .. (10.93,3.29)   ;
  \draw    (100,101) -- (148.27,129) ;
  \draw [shift={(150,130)}, rotate = 210.11] [color={rgb, 255:red, 0; green, 0; blue, 0 }  ][line width=0.75]    (10.93,-3.29) .. controls (6.95,-1.4) and (3.31,-0.3) .. (0,0) .. controls (3.31,0.3) and (6.95,1.4) .. (10.93,3.29)   ;
  \draw    (100,101) -- (109.6,148.04) ;
  \draw [shift={(110,150)}, rotate = 258.47] [color={rgb, 255:red, 0; green, 0; blue, 0 }  ][line width=0.75]    (10.93,-3.29) .. controls (6.95,-1.4) and (3.31,-0.3) .. (0,0) .. controls (3.31,0.3) and (6.95,1.4) .. (10.93,3.29)   ;
  \draw    (100,101) -- (61.62,128.83) ;
  \draw [shift={(60,130)}, rotate = 324.06] [color={rgb, 255:red, 0; green, 0; blue, 0 }  ][line width=0.75]    (10.93,-3.29) .. controls (6.95,-1.4) and (3.31,-0.3) .. (0,0) .. controls (3.31,0.3) and (6.95,1.4) .. (10.93,3.29)   ;
  \draw    (230,50) .. controls (256.25,78.5) and (307.25,76) .. (330,50) ;
  \draw    (320,40) .. controls (292.25,63.5) and (290.25,108) .. (320,140) ;
  \draw    (340,120) .. controls (310.25,109) and (276.75,118.5) .. (260,150) ;
  \draw    (280,150) .. controls (273.75,129) and (231.75,108) .. (210,120) ;
  
  \draw (111,52.4) node [anchor=north west][inner sep=0.75pt]    {$\sigma _{1}$};
  \draw (131,81.4) node [anchor=north west][inner sep=0.75pt]    {$\sigma _{2}$};
  \draw (81,111.4) node [anchor=north west][inner sep=0.75pt]    {$\sigma _{l}$};
  \draw (101,22.4) node [anchor=north west][inner sep=0.75pt]    {$v_{1}$};
  \draw (141,42.4) node [anchor=north west][inner sep=0.75pt]    {$v_{2}$};
  \draw (151,121.4) node [anchor=north west][inner sep=0.75pt]    {$v_{3}$};
  \draw (102,151.4) node [anchor=north west][inner sep=0.75pt]    {$v_{l}$};
  \draw (41,130.4) node [anchor=north west][inner sep=0.75pt]    {$v_{l+1}$};
  \draw (101,113.4) node [anchor=north west][inner sep=0.75pt]    {$\cdots $};
  \draw (261,52.4) node [anchor=north west][inner sep=0.75pt]    {$D_{1}$};
  \draw (301,81.4) node [anchor=north west][inner sep=0.75pt]    {$D_{2}$};
  \draw (226,121.4) node [anchor=north west][inner sep=0.75pt]    {$D_{l+1}$};
  \draw (281,123.4) node [anchor=north west][inner sep=0.75pt]    {$\cdots $};

  \end{tikzpicture}
  
  \end{center}    

Consider the values of $a_i$, which are derived from the self-intersection of divisors. Let $X_{\Sigma}$ denote a smooth and compact toric variety of dimension $n$, corresponding to a regular and complete fan $\Sigma$. Denote the 1-cones of $\Sigma$ as $\rho_1,...,\rho_n$ and the corresponding divisors as $D_{\rho_1},...,D_{\rho_n}$. If $\sigma:=\rho_1+\cdots+\rho_n$ constitutes an $n$-cone of $\Sigma$, then the divisors $D_{\rho_i}$ serve as closures of the coordinate hyperplanes of $X_{\check{\sigma}}$, resulting in a transversal intersection at a single point. Conversely, if the sum $\rho_1+\cdots+\rho_n$ does not lie within a cell of $\Sigma$, the intersection of the subvarieties $D_{\rho_i}$ will be empty, yielding an intersection number of zero.

\begin{lemma}\cite[Lemma 6.2]{ewald1996combinatorial}
    For $n=2$, let $\tau$ be a 1-cone of a regular complete fan $\Sigma$ defined by $\tau = \mathbb{R}_{\geq 0} a$. Let $\sigma_1 := \mathbb{R}_{\geq 0} a + \mathbb{R}_{\geq 0} b$ and $\sigma_2 := \mathbb{R}_{\geq 0} a + \mathbb{R}_{\geq 0} c$ denote the adjacent 2-cones of $\tau$ within $\Sigma$, where $a, b, c$ are simple. If $\alpha \in \mathbb{Z}$ is the integer for which
    \begin{equation*}
        \alpha a + b + c = 0,
    \end{equation*}
    then, the self-intersection number of $D_{\tau}$ is given by
    \begin{equation*}
        (D_{\tau}, D_{\tau}) = \alpha.
    \end{equation*}
\end{lemma}

\begin{example}\label{P2}
    The projective plane $\mathbb{P}^2$ is represented by a fan with generators $a = e_1$, $b = e_2$, and $c = -e_1 - e_2$. As such, we have the relation $a + b + c = 0$. According to the aforementioned lemma, the self-intersection number $(D_{\tau}, D_{\tau}) = 1$ for each coordinate line of $\mathbb{P}^2$. This outcome arises from the fact that any two distinct lines in $\mathbb{P}^2$ are projectively equivalent and intersect at a single point.
\end{example}

\begin{example}\label{blowdown}
    If $ a = b + c $, we can derive the fan $ \Sigma' $ from $ \Sigma $ by removing the cones $ \tau, \sigma_1, \sigma_2 $ and introducing a new cone formed by $ \sigma_1 \cup \sigma_2 $. Consequently, the variety $ X_{\Sigma} $ is obtained from $ X_{\Sigma'} $ through a blowup, where the exceptional projective line exhibits a self-intersection number of -1.
\end{example}

\begin{remark}
  Let $ D $ be any divisor on the surface $ X $. We can define the self-intersection number $ (D, D) $. For a nonsingular curve $ C $ on $ X $, the self-intersection number is expressed as $ (C, C) = \deg_C(\mathcal{L}(C) \otimes \mathcal{O}_C) $. This can also be interpreted through the fact that the ideal sheaf $ \mathcal{I} $ of $ C $ on $ X $ is given by $ \mathcal{L}(-C) $. Consequently, we have $ \mathcal{I}/\mathcal{I}^2 \cong \mathcal{L}(-C) \otimes \mathcal{O}_C $. Thus, its dual, $ \mathcal{L}(C) \otimes \mathcal{O}_C $, is isomorphic to the normal sheaf $ \mathcal{N}_{C/X} $, which is defined as $ \mathcal{H}om(\mathcal{I}/\mathcal{I}^2, \mathcal{O}_C) $. Therefore, it follows that $ (C, C) = \deg_{C} \mathcal{N}_{C/X} $.
\end{remark}

Through the construction of the cellular complex, the self-intersection numbers $ (D_{v_i}, D_{v_i}) $ provide significant insights into the cellular structure of $ X_{\Sigma} $ as these values are intrinsically linked to the corresponding normal bundles.

To compute the cellular homology, we can assign orientations to the cellular complex: we can assign the compatible orientation to $ v_i $ that originates from $ \sigma_i $. Subsequently, we can compare the two orientations of $ v_i $, which arise from $ \sigma_{i-1} $ and $ \sigma_i $.

The following lemma directly follows from Proposition \ref{tActSigma}:

\begin{lemma}\label{changeOriD2}
  Let $ \iota_{\sigma_i}, \iota_{\sigma_{i-1}}: \mathrm{Th}_{X_{v_i}}(Z_{v_i}) \xrightarrow{\cong} \af^1/\afnz{1} \wedge \mathbb{G}_m $ be the trivializations induced by the orientations from $ \sigma_{i} $ and $ \sigma_{i-1} $. These isomorphisms give rise to an automorphism:
  \[
    \tau_i: \KMW_2 \oplus \KMW_1 \xrightarrow{\cong} \KMW_2 \oplus \KMW_1 
  \]
  which is represented by the matrix $ \begin{bmatrix}
  \aBra{-1}^{a_i} & (a_i)_{\epsilon}\eta \\
  0 & 1
\end{bmatrix} $.
\end{lemma}

With this lemma established, we can proceed to discuss the cellular complex of a complete toric surface $ X_{\Sigma} $. 

\subsection{Complete Toric Surfaces}

\begin{proposition}\label{examplesurface}
  Let $\Sigma$ be a $2$-dimensional complete fan consisting of $l$ cones and $l$ edges. The cellular complex $\wt{C}^{cell}_*(X_{\Sigma})$ can be expressed as follows: 

\[\begin{tikzcd}[ampersand replacement=\&,cramped]
	{ (\KMW_2)_{\sigma_1}} \& { (\KMW_1 \otimes \mathbf{H})_{v_1}} \\
	{ (\KMW_2)_{\sigma_2}} \& { (\KMW_1 \otimes \mathbf{H})_{v_2}} \& {\mathbf{H} \otimes \mathbf{H}} \\
	\vdots \& \vdots \\
	{ (\KMW_2)_{\sigma_l}} \& { (\KMW_1 \otimes \mathbf{H})_{v_l}}
	\arrow["{(\mathrm{Id}, 0)}", from=1-1, to=1-2]
	\arrow["{d( a_{2} )}"{description}, from=1-1, to=2-2]
	\arrow[from=1-2, to=2-3]
	\arrow["\bigoplus"{description}, draw=none, from=2-1, to=1-1]
	\arrow["{(\mathrm{Id}, 0)}", from=2-1, to=2-2]
	\arrow["{d( a_3 )}"{description}, from=2-1, to=3-2]
	\arrow["\bigoplus"{description}, draw=none, from=2-2, to=1-2]
	\arrow[from=2-2, to=2-3]
	\arrow["\bigoplus"{description}, draw=none, from=3-1, to=2-1]
	\arrow["{(\mathrm{Id}, 0)}", from=3-1, to=3-2]
	\arrow["{d( a_{l} )}"{description}, from=3-1, to=4-2]
	\arrow["\bigoplus"{description}, draw=none, from=3-2, to=2-2]
	\arrow[from=3-2, to=2-3]
	\arrow["\bigoplus"{description}, draw=none, from=4-1, to=3-1]
	\arrow["{(\mathrm{Id}, 0)}", from=4-1, to=4-2]
	\arrow[from=4-2, to=2-3]
	\arrow["\bigoplus"{description}, draw=none, from=4-2, to=3-2]
	\arrow["{d( a_{1} )}"{description, pos=0.1}, from=4-1, to=1-2]
\end{tikzcd}\]
where the morphism $d(a_i)$ is defined as $d(a_i) := (\aBra{-1}^{a_i}, (a_i)_{\epsilon}\eta)$. Furthermore, let $ a_{\Sigma} = \mathrm{gcd}(a_1,\ldots,a_l) $. This complex is quasi-isomorphic to the following:

\[\begin{tikzcd}[ampersand replacement=\&,cramped]
	{\KMW_2} \& {\KMW_1} \\
	\& {(\KMW_1)^{\oplus l-3}} \& \bZ
	\arrow["{(a_{\Sigma})_{\epsilon}\eta}", from=1-1, to=1-2]
	\arrow["0"', from=1-1, to=2-2]
	\arrow["0", from=1-2, to=2-3]
	\arrow["\bigoplus"', draw=none, from=2-2, to=1-2]
	\arrow["0"', from=2-2, to=2-3]
\end{tikzcd}\]

\end{proposition}

\begin{corollary}
  The motivic decomposition for $\M(X_{\Sigma}) \in \DM(k)$ is given by:
  \[
    \M(X_\Sigma) \cong \bZ(2)[4] \oplus (\bZ(1)[2])^{\oplus l-2} \oplus \bZ.
  \]
\end{corollary}

\begin{corollary}
  The MW-motivic decomposition for $\Mt(X_{\Sigma}) \in \DMt(k)$ is described as follows:
  \[
      \Mt(X_\Sigma) \cong \begin{cases}
              \tbZ(2)[4] \oplus (\tbZ(1)[2])^{\oplus l-2} \oplus \tbZ, & \text{if } a_{\Sigma} \text{ is even,} \\
              
              \tbZ(1)[2]\sslash \eta \oplus (\tbZ(1)[2])^{\oplus l-3} \oplus \tbZ, & \text{if } a_{\Sigma} \text{ is odd.} \\
          \end{cases}    
\]

\end{corollary}

\begin{remark}
  This result directly parallels the computations concerning orientable and non-orientable surfaces in topology. The orientability is contingent upon the parity of $a_{\Sigma}$, and we can retrieve the homology of surfaces by considering the real realization (let $\eta = 2$).
\end{remark}

Before presenting the proof, we will compute several examples manually to elucidate these concepts.

\begin{example}[Projective Space]
  For $\mathbb{P}^2$, we examine the following fan and dual cones. Each cone corresponds to an affine toric variety, and we highlight the generators of the coordinate ring, which contain the information regarding orientation.
    
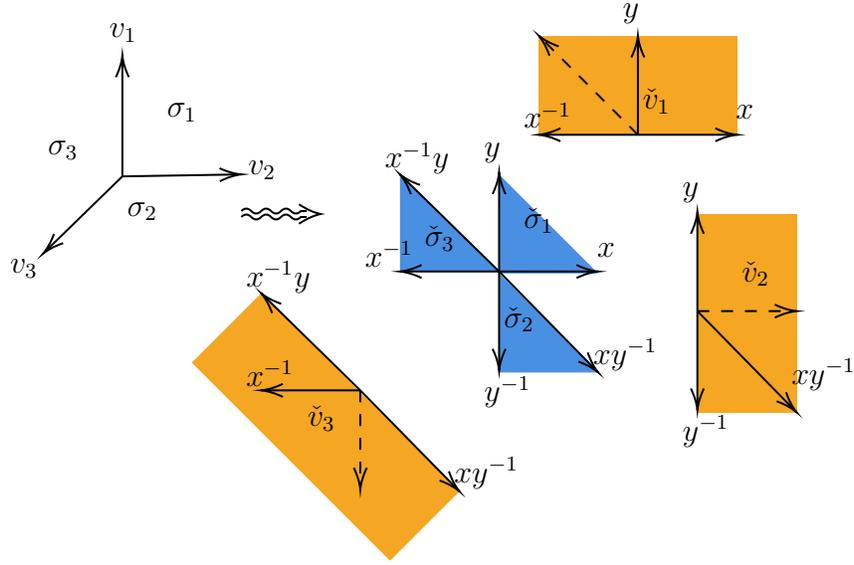
\begin{center}

\tikzset{every picture/.style={line width=0.75pt}} 

\begin{tikzpicture}[x=0.75pt,y=0.75pt,yscale=-1,xscale=1]

\draw  [draw=none][fill={rgb, 255:red, 245; green, 166; blue, 35 }  ,fill opacity=0.5 ] (310,10) -- (410,10) -- (410,60) -- (310,60) -- cycle ;
\draw  [draw=none][fill={rgb, 255:red, 245; green, 166; blue, 35 }  ,fill opacity=0.5 ] (390,100) -- (440,100) -- (440,200) -- (390,200) -- cycle ;
\draw  [draw=none][fill={rgb, 255:red, 245; green, 166; blue, 35 }  ,fill opacity=0.5 ] (135,175) -- (170,140) -- (270,240) -- (235,275) -- cycle ;
\draw  [draw=none][fill={rgb, 255:red, 74; green, 144; blue, 226 }  ,fill opacity=0.5 ] (290,80) -- (340,130) -- (290,130) -- cycle ;
\draw  [draw=none][fill={rgb, 255:red, 74; green, 144; blue, 226 }  ,fill opacity=0.5 ] (240,80) -- (290,129) -- (240,129) -- cycle ;
\draw  [draw=none][fill={rgb, 255:red, 74; green, 144; blue, 226 }  ,fill opacity=0.5 ] (290,130) -- (340,180) -- (290,180) -- cycle ;
\draw    (100,81) -- (100,22) ;
\draw [shift={(100,20)}, rotate = 90] [color={rgb, 255:red, 0; green, 0; blue, 0 }  ][line width=0.75]    (10.93,-3.29) .. controls (6.95,-1.4) and (3.31,-0.3) .. (0,0) .. controls (3.31,0.3) and (6.95,1.4) .. (10.93,3.29)   ;
\draw    (100,81) -- (158,80.03) ;
\draw [shift={(160,80)}, rotate = 179.05] [color={rgb, 255:red, 0; green, 0; blue, 0 }  ][line width=0.75]    (10.93,-3.29) .. controls (6.95,-1.4) and (3.31,-0.3) .. (0,0) .. controls (3.31,0.3) and (6.95,1.4) .. (10.93,3.29)   ;
\draw    (100,81) -- (61.43,118.6) ;
\draw [shift={(60,120)}, rotate = 315.73] [color={rgb, 255:red, 0; green, 0; blue, 0 }  ][line width=0.75]    (10.93,-3.29) .. controls (6.95,-1.4) and (3.31,-0.3) .. (0,0) .. controls (3.31,0.3) and (6.95,1.4) .. (10.93,3.29)   ;
\draw    (290,129) -- (290,81) ;
\draw [shift={(290,79)}, rotate = 90] [color={rgb, 255:red, 0; green, 0; blue, 0 }  ][line width=0.75]    (10.93,-3.29) .. controls (6.95,-1.4) and (3.31,-0.3) .. (0,0) .. controls (3.31,0.3) and (6.95,1.4) .. (10.93,3.29)   ;
\draw    (290,129) -- (338,129) ;
\draw [shift={(340,129)}, rotate = 180] [color={rgb, 255:red, 0; green, 0; blue, 0 }  ][line width=0.75]    (10.93,-3.29) .. controls (6.95,-1.4) and (3.31,-0.3) .. (0,0) .. controls (3.31,0.3) and (6.95,1.4) .. (10.93,3.29)   ;
\draw    (290,129) -- (241.43,81.4) ;
\draw [shift={(240,80)}, rotate = 44.42] [color={rgb, 255:red, 0; green, 0; blue, 0 }  ][line width=0.75]    (10.93,-3.29) .. controls (6.95,-1.4) and (3.31,-0.3) .. (0,0) .. controls (3.31,0.3) and (6.95,1.4) .. (10.93,3.29)   ;
\draw    (290,129) -- (338.6,178.57) ;
\draw [shift={(340,180)}, rotate = 225.57] [color={rgb, 255:red, 0; green, 0; blue, 0 }  ][line width=0.75]    (10.93,-3.29) .. controls (6.95,-1.4) and (3.31,-0.3) .. (0,0) .. controls (3.31,0.3) and (6.95,1.4) .. (10.93,3.29)   ;
\draw    (290,129) -- (242,129) ;
\draw [shift={(240,129)}, rotate = 360] [color={rgb, 255:red, 0; green, 0; blue, 0 }  ][line width=0.75]    (10.93,-3.29) .. controls (6.95,-1.4) and (3.31,-0.3) .. (0,0) .. controls (3.31,0.3) and (6.95,1.4) .. (10.93,3.29)   ;
\draw    (290,129) -- (290,177) ;
\draw [shift={(290,179)}, rotate = 270] [color={rgb, 255:red, 0; green, 0; blue, 0 }  ][line width=0.75]    (10.93,-3.29) .. controls (6.95,-1.4) and (3.31,-0.3) .. (0,0) .. controls (3.31,0.3) and (6.95,1.4) .. (10.93,3.29)   ;
\draw    (360,60) -- (360,12) ;
\draw [shift={(360,10)}, rotate = 90] [color={rgb, 255:red, 0; green, 0; blue, 0 }  ][line width=0.75]    (10.93,-3.29) .. controls (6.95,-1.4) and (3.31,-0.3) .. (0,0) .. controls (3.31,0.3) and (6.95,1.4) .. (10.93,3.29)   ;
\draw    (360,60) -- (408,60) ;
\draw [shift={(410,60)}, rotate = 180] [color={rgb, 255:red, 0; green, 0; blue, 0 }  ][line width=0.75]    (10.93,-3.29) .. controls (6.95,-1.4) and (3.31,-0.3) .. (0,0) .. controls (3.31,0.3) and (6.95,1.4) .. (10.93,3.29)   ;
\draw  [dash pattern={on 4.5pt off 4.5pt}]  (360,60) -- (311.41,11.41) ;
\draw [shift={(310,10)}, rotate = 45] [color={rgb, 255:red, 0; green, 0; blue, 0 }  ][line width=0.75]    (10.93,-3.29) .. controls (6.95,-1.4) and (3.31,-0.3) .. (0,0) .. controls (3.31,0.3) and (6.95,1.4) .. (10.93,3.29)   ;
\draw    (360,60) -- (312,60) ;
\draw [shift={(310,60)}, rotate = 360] [color={rgb, 255:red, 0; green, 0; blue, 0 }  ][line width=0.75]    (10.93,-3.29) .. controls (6.95,-1.4) and (3.31,-0.3) .. (0,0) .. controls (3.31,0.3) and (6.95,1.4) .. (10.93,3.29)   ;
\draw    (390,149) -- (390,101) ;
\draw [shift={(390,99)}, rotate = 90] [color={rgb, 255:red, 0; green, 0; blue, 0 }  ][line width=0.75]    (10.93,-3.29) .. controls (6.95,-1.4) and (3.31,-0.3) .. (0,0) .. controls (3.31,0.3) and (6.95,1.4) .. (10.93,3.29)   ;
\draw  [dash pattern={on 4.5pt off 4.5pt}]  (390,149) -- (438,149) ;
\draw [shift={(440,149)}, rotate = 180] [color={rgb, 255:red, 0; green, 0; blue, 0 }  ][line width=0.75]    (10.93,-3.29) .. controls (6.95,-1.4) and (3.31,-0.3) .. (0,0) .. controls (3.31,0.3) and (6.95,1.4) .. (10.93,3.29)   ;
\draw    (390,149) -- (438.6,198.57) ;
\draw [shift={(440,200)}, rotate = 225.57] [color={rgb, 255:red, 0; green, 0; blue, 0 }  ][line width=0.75]    (10.93,-3.29) .. controls (6.95,-1.4) and (3.31,-0.3) .. (0,0) .. controls (3.31,0.3) and (6.95,1.4) .. (10.93,3.29)   ;
\draw    (390,149) -- (390,197) ;
\draw [shift={(390,199)}, rotate = 270] [color={rgb, 255:red, 0; green, 0; blue, 0 }  ][line width=0.75]    (10.93,-3.29) .. controls (6.95,-1.4) and (3.31,-0.3) .. (0,0) .. controls (3.31,0.3) and (6.95,1.4) .. (10.93,3.29)   ;
\draw    (220,189) -- (171.43,141.4) ;
\draw [shift={(170,140)}, rotate = 44.42] [color={rgb, 255:red, 0; green, 0; blue, 0 }  ][line width=0.75]    (10.93,-3.29) .. controls (6.95,-1.4) and (3.31,-0.3) .. (0,0) .. controls (3.31,0.3) and (6.95,1.4) .. (10.93,3.29)   ;
\draw    (220,189) -- (268.6,238.57) ;
\draw [shift={(270,240)}, rotate = 225.57] [color={rgb, 255:red, 0; green, 0; blue, 0 }  ][line width=0.75]    (10.93,-3.29) .. controls (6.95,-1.4) and (3.31,-0.3) .. (0,0) .. controls (3.31,0.3) and (6.95,1.4) .. (10.93,3.29)   ;
\draw    (220,189) -- (172,189) ;
\draw [shift={(170,189)}, rotate = 360] [color={rgb, 255:red, 0; green, 0; blue, 0 }  ][line width=0.75]    (10.93,-3.29) .. controls (6.95,-1.4) and (3.31,-0.3) .. (0,0) .. controls (3.31,0.3) and (6.95,1.4) .. (10.93,3.29)   ;
\draw  [dash pattern={on 4.5pt off 4.5pt}]  (220,189) -- (220,237) ;
\draw [shift={(220,239)}, rotate = 270] [color={rgb, 255:red, 0; green, 0; blue, 0 }  ][line width=0.75]    (10.93,-3.29) .. controls (6.95,-1.4) and (3.31,-0.3) .. (0,0) .. controls (3.31,0.3) and (6.95,1.4) .. (10.93,3.29)   ;
\draw    (160,98.5) .. controls (161.67,96.83) and (163.33,96.83) .. (165,98.5) .. controls (166.67,100.17) and (168.33,100.17) .. (170,98.5) .. controls (171.67,96.83) and (173.33,96.83) .. (175,98.5) .. controls (176.67,100.17) and (178.33,100.17) .. (180,98.5) .. controls (181.67,96.83) and (183.33,96.83) .. (185,98.5) .. controls (186.67,100.17) and (188.33,100.17) .. (190,98.5) -- (193,98.5)(160,101.5) .. controls (161.67,99.83) and (163.33,99.83) .. (165,101.5) .. controls (166.67,103.17) and (168.33,103.17) .. (170,101.5) .. controls (171.67,99.83) and (173.33,99.83) .. (175,101.5) .. controls (176.67,103.17) and (178.33,103.17) .. (180,101.5) .. controls (181.67,99.83) and (183.33,99.83) .. (185,101.5) .. controls (186.67,103.17) and (188.33,103.17) .. (190,101.5) -- (193,101.5) ;
\draw [shift={(200,100)}, rotate = 180] [color={rgb, 255:red, 0; green, 0; blue, 0 }  ][line width=0.75]    (10.93,-4.9) .. controls (6.95,-2.3) and (3.31,-0.67) .. (0,0) .. controls (3.31,0.67) and (6.95,2.3) .. (10.93,4.9)   ;

\draw (121,42.4) node [anchor=north west][inner sep=0.75pt]    {$\sigma _{1}$};
\draw (101,92.4) node [anchor=north west][inner sep=0.75pt]    {$\sigma _{2}$};
\draw (61,61.4) node [anchor=north west][inner sep=0.75pt]    {$\sigma _{3}$};
\draw (92,2.4) node [anchor=north west][inner sep=0.75pt]    {$v_{1}$};
\draw (162,71.4) node [anchor=north west][inner sep=0.75pt]    {$v_{2}$};
\draw (42,121.4) node [anchor=north west][inner sep=0.75pt]    {$v_{3}$};
\draw (301,95.4) node [anchor=north west][inner sep=0.75pt]    {$\check{\sigma }_{1}$};
\draw (291,145.4) node [anchor=north west][inner sep=0.75pt]    {$\check{\sigma }_{2}$};
\draw (251,104.4) node [anchor=north west][inner sep=0.75pt]    {$\check{\sigma }_{3}$};
\draw (338,113.4) node [anchor=north west][inner sep=0.75pt]    {$x$};
\draw (281,62.4) node [anchor=north west][inner sep=0.75pt]    {$y$};
\draw (221,112.4) node [anchor=north west][inner sep=0.75pt]    {$x^{-1}$};
\draw (281,180.4) node [anchor=north west][inner sep=0.75pt]    {$y^{-1}$};
\draw (336,161.4) node [anchor=north west][inner sep=0.75pt]    {$xy^{-1}$};
\draw (231,62.4) node [anchor=north west][inner sep=0.75pt]    {$x^{-1} y$};
\draw (361,35.4) node [anchor=north west][inner sep=0.75pt]    {$\check{v}_{1}$};
\draw (408,43.4) node [anchor=north west][inner sep=0.75pt]    {$x$};
\draw (351,-7.6) node [anchor=north west][inner sep=0.75pt]    {$y$};
\draw (301,42.4) node [anchor=north west][inner sep=0.75pt]    {$x^{-1}$};
\draw (381,82.4) node [anchor=north west][inner sep=0.75pt]    {$y$};
\draw (381,200.4) node [anchor=north west][inner sep=0.75pt]    {$y^{-1}$};
\draw (411,122.4) node [anchor=north west][inner sep=0.75pt]    {$\check{v}_{2}$};
\draw (436,171.4) node [anchor=north west][inner sep=0.75pt]    {$xy^{-1}$};
\draw (161,172.4) node [anchor=north west][inner sep=0.75pt]    {$x^{-1}$};
\draw (266,221.4) node [anchor=north west][inner sep=0.75pt]    {$xy^{-1}$};
\draw (161,122.4) node [anchor=north west][inner sep=0.75pt]    {$x^{-1} y$};
\draw (192,195.4) node [anchor=north west][inner sep=0.75pt]    {$\check{v}_{3}$};

\end{tikzpicture}

\end{center}

We now translate this information into coordinate rings: 

\[\begin{tikzcd}[ampersand replacement=\&,cramped]
	\& {k[x^{-1}y,x,x^{-1}]} \\
	{k[\check{\sigma}_1]\cong k[x,y]} \& {k[y,x,x^{-1}]\cong k[\check{v}_1]} \\
	\& { k[x,y,y^{-1}]} \\
	{k[\check{\sigma}_2] \cong k[xy^{-1},y^{-1}]} \& {k[xy^{-1},y,y^{-1}] \cong k[\check{v}_2]} \& {k[x,x^{-1},y,y^{-1}]} \\
	\& {k[y^{-1},xy^{-1},x^{-1}y]} \\
	{k[\check{\sigma}_3] \cong k[x^{-1},x^{-1}y]} \& {k[x^{-1},xy^{-1},x^{-1}y]\cong k[\check{v}_3]}
	\arrow["\cong", from=1-2, to=2-2]
	\arrow[no head, from=2-1, to=2-2]
	\arrow[no head, from=2-1, to=3-2]
	\arrow[no head, from=2-2, to=4-3]
	\arrow["\cong", from=3-2, to=4-2]
	\arrow[no head, from=4-1, to=4-2]
	\arrow[no head, from=4-1, to=5-2]
	\arrow[no head, from=4-2, to=4-3]
	\arrow["\cong", from=5-2, to=6-2]
	\arrow[no head, from=6-1, to=1-2]
	\arrow[no head, from=6-1, to=6-2]
	\arrow[no head, from=6-2, to=4-3]
\end{tikzcd}\]

It is noteworthy that the isomorphism $ k[x^{-1}y,x,x^{-1}] \xr{\cong} k[y,x,x^{-1}]$ induces a morphism $ \af^1 / \afnz{1} \wedge \Gm \xr{\cong} \af^1 / \afnz{1} \wedge \Gm $, thereby producing a morphism $ \KMW_1 \otimes \mathbf{H} \to \KMW_1 \otimes \mathbf{H} $. According to Proposition \ref{tActMain} (also Lemma \ref{changeOriD2}), this morphism can be expressed as:

\[
\begin{bmatrix}
  \aBra{-1} & \eta \\
  0 & 1
\end{bmatrix}
\]
In this case, the cellular complex takes the following form:
\[\begin{tikzcd}[ampersand replacement=\&,cramped]
	{ (\KMW_2)_{\sigma_1}} \& { (\KMW_1 \otimes \mathbf{H})_{v_1}} \\
	{ (\KMW_2)_{\sigma_2}} \& { (\KMW_1 \otimes \mathbf{H})_{v_2}} \& {\mathbf{H} \otimes \mathbf{H}} \\
	{ (\KMW_2)_{\sigma_3}} \& { (\KMW_1 \otimes \mathbf{H})_{v_3}}
	\arrow["{(\mathrm{Id}, 0)}", from=1-1, to=1-2]
	\arrow["{d( 1 )}"{description}, from=1-1, to=2-2]
	\arrow[from=1-2, to=2-3]
	\arrow["\bigoplus"{description}, draw=none, from=2-1, to=1-1]
	\arrow["{(\mathrm{Id}, 0)}", from=2-1, to=2-2]
	\arrow["{d( 1 )}"{description}, from=2-1, to=3-2]
	\arrow["\bigoplus"{description}, draw=none, from=2-2, to=1-2]
	\arrow[from=2-2, to=2-3]
	\arrow["{d( 1 )}"{description, pos=0.2}, from=3-1, to=1-2]
	\arrow["\bigoplus"{description}, draw=none, from=3-1, to=2-1]
	\arrow["{(\mathrm{Id}, 0)}", from=3-1, to=3-2]
	\arrow["\bigoplus"{description}, draw=none, from=3-2, to=2-2]
	\arrow[from=3-2, to=2-3]
\end{tikzcd}\]

After performing some reductions, this cellular complex becomes quasi-isomorphic to:
\[
  \wt{C}^{cell}_*(\mathbb{P}^2) \cong \KMW_2 \xr{\eta} \KMW_1 \xr{0} \bZ
\]
\end{example}

\begin{example}[Hirzebruch surfaces]
\label{Hirzebruch}
The Hirzebruch surface, denoted as $\mathbb{F}_a$, is a toric surface defined by the associated fan:

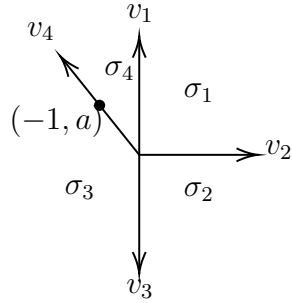
\begin{center}

\tikzset{every picture/.style={line width=0.75pt}} 

\begin{tikzpicture}[x=0.75pt,y=0.75pt,yscale=-1,xscale=1]

\draw    (120,110) -- (120,52) ;
\draw [shift={(120,50)}, rotate = 90] [color={rgb, 255:red, 0; green, 0; blue, 0 }  ][line width=0.75]    (10.93,-3.29) .. controls (6.95,-1.4) and (3.31,-0.3) .. (0,0) .. controls (3.31,0.3) and (6.95,1.4) .. (10.93,3.29)   ;
\draw    (120,110) -- (178,110) ;
\draw [shift={(180,110)}, rotate = 180] [color={rgb, 255:red, 0; green, 0; blue, 0 }  ][line width=0.75]    (10.93,-3.29) .. controls (6.95,-1.4) and (3.31,-0.3) .. (0,0) .. controls (3.31,0.3) and (6.95,1.4) .. (10.93,3.29)   ;
\draw    (120,110) -- (120,168) ;
\draw [shift={(120,170)}, rotate = 270] [color={rgb, 255:red, 0; green, 0; blue, 0 }  ][line width=0.75]    (10.93,-3.29) .. controls (6.95,-1.4) and (3.31,-0.3) .. (0,0) .. controls (3.31,0.3) and (6.95,1.4) .. (10.93,3.29)   ;
\draw    (120,110) -- (81.25,61.56) ;
\draw [shift={(80,60)}, rotate = 51.34] [color={rgb, 255:red, 0; green, 0; blue, 0 }  ][line width=0.75]    (10.93,-3.29) .. controls (6.95,-1.4) and (3.31,-0.3) .. (0,0) .. controls (3.31,0.3) and (6.95,1.4) .. (10.93,3.29)   ;
\draw  [fill={rgb, 255:red, 0; green, 0; blue, 0 }  ,fill opacity=1 ] (97.68,85) .. controls (97.68,86.28) and (98.72,87.32) .. (100,87.32) .. controls (101.28,87.32) and (102.32,86.28) .. (102.32,85) .. controls (102.32,83.72) and (101.28,82.68) .. (100,82.68) .. controls (98.72,82.68) and (97.68,83.72) .. (97.68,85) -- cycle ;

\draw (141,72.4) node [anchor=north west][inner sep=0.75pt]    {$\sigma _{1}$};
\draw (141,122.4) node [anchor=north west][inner sep=0.75pt]    {$\sigma _{2}$};
\draw (81,121.4) node [anchor=north west][inner sep=0.75pt]    {$\sigma _{3}$};
\draw (112,32.4) node [anchor=north west][inner sep=0.75pt]    {$v_{1}$};
\draw (182,101.4) node [anchor=north west][inner sep=0.75pt]    {$v_{2}$};
\draw (112,171.4) node [anchor=north west][inner sep=0.75pt]    {$v_{3}$};
\draw (53,83.4) node [anchor=north west][inner sep=0.75pt]    {$( -1,a)$};
\draw (101,61.4) node [anchor=north west][inner sep=0.75pt]    {$\sigma _{4}$};
\draw (62,41.4) node [anchor=north west][inner sep=0.75pt]    {$v_{4}$};

\end{tikzpicture}

\end{center}

In line with our previous considerations, we can manually compute the transform morphism for different orientations. It is important to note that in this instance, we have $a_1=a$, $a_2=0$, $a_3=-a$, and $a_4=0$. Consequently, the cellular complex can be expressed in the following form:

\[\begin{tikzcd}[ampersand replacement=\&,cramped]
	{ (\KMW_2)_{\sigma_1}} \& { (\KMW_1 \otimes \mathbf{H})_{v_1}} \\
	{ (\KMW_2)_{\sigma_2}} \& { (\KMW_1 \otimes \mathbf{H})_{v_2}} \& {\mathbf{H} \otimes \mathbf{H}} \\
	{ (\KMW_2)_{\sigma_3}} \& { (\KMW_1 \otimes \mathbf{H})_{v_3}} \\
	{ (\KMW_2)_{\sigma_4}} \& { (\KMW_1 \otimes \mathbf{H})_{v_4}}
	\arrow["{(\mathrm{Id}, 0)}", from=1-1, to=1-2]
	\arrow["{(\mathrm{Id}, 0)}"{description}, from=1-1, to=2-2]
	\arrow[from=1-2, to=2-3]
	\arrow["\bigoplus"{description}, draw=none, from=2-1, to=1-1]
	\arrow["{(\mathrm{Id}, 0)}", from=2-1, to=2-2]
	\arrow["{d( a )}"{description, pos=0.3}, from=2-1, to=3-2]
	\arrow["\bigoplus"{description}, draw=none, from=2-2, to=1-2]
	\arrow[from=2-2, to=2-3]
	\arrow["\bigoplus"{description}, draw=none, from=3-1, to=2-1]
	\arrow["{(\mathrm{Id}, 0)}", from=3-1, to=3-2]
	\arrow["{(\mathrm{Id}, 0)}"{description}, from=3-1, to=4-2]
	\arrow["\bigoplus"{description}, draw=none, from=3-2, to=2-2]
	\arrow[from=3-2, to=2-3]
	\arrow["{d( a )}"{description, pos=0.1}, from=4-1, to=1-2]
	\arrow["\bigoplus"{description}, draw=none, from=4-1, to=3-1]
	\arrow["{(\mathrm{Id}, 0)}", from=4-1, to=4-2]
	\arrow[from=4-2, to=2-3]
	\arrow["\bigoplus"{description}, draw=none, from=4-2, to=3-2]
\end{tikzcd}\]

Based on the identification established in Theorem \ref{mainComplex}, we can construct the cellular complex as follows:

Let $ K $ denote the boundary complex of a Hirzebruch surface $ \mathbb{F}_a $. This complex has a shelling represented by
\[ 12, 2\underline{3}, 3\underline{4}, \underline{14}, \]
with restrictions indicated accordingly. Define a characteristic function $ \lambda: \mathbb{Z}^4 \to \mathbb{Z}^2 $ as represented by the following bordermatrix:
\begin{equation*}
    \lambda = \bordermatrix{
     & 1 & 2 & 3 & 4 \cr
     & 0 & 1 & 0 & -1\cr
     & 1 & 0 & -1 & a\cr
     }
\end{equation*}
This function satisfies the non-degenerate condition over the complex $ K $, where the values above the matrix represent the indices of the vertices of $ K $. It can be observed that the moment complex $ \mathcal{A}Z_K $ is isomorphic to $ \mathbb{P}^1 \times \mathbb{P}^1 $.

For the case when $ a $ is odd, we have $ \omega \in \{\emptyset, 134, 24, 123\} $. According to the rule $ r(\sigma) = \omega \cap \sigma $ for $ \sigma \in K_{\text{max}} $, we find that $ r(12) = \emptyset \in K_{\emptyset} $, $ r(34) = 4 \in K_{24} $, while $ r(23) = 3 $ and $ r(14) = 14 $ belong to $ K_{134} $ (with no restrictions in $ K_{123} $). The defined shelling provides a regular expanding sequence for $ K_{134} $, specifically $ 1, 3, 34, 14 $, as established in \cite[Proposition 4.5]{cai2021integral}, where both 3 and 14 are marked as critical. The simplicial retraction yields $ \rho([34]) = 0 $ and $ \rho([4]) = [3] $, as outlined in \cite[Lemma 4.2 and 4.4]{cai2021integral}. Consequently, we obtain the following relation:
\begin{equation*}
    \overline{\partial}^{\text{cri}}([14]) = [3].
\end{equation*}
This leads to the formation of the following complex, as dictated by Theorem \ref{mainComplex}:
\[
\wt{C}^{cell}_*(\mathbb{F}_a) \cong \KMW_2 \xr{(\eta,0)} \KMW_1 \oplus \KMW_1 \xr{0} \mathbb{Z}.
\]

In the case where $ a $ is even, we then have $ \omega \in \{\emptyset, 13, 24, 1234\} $. Under these conditions, it follows that $ r(12) = \emptyset \in K_{\emptyset} $, $ r(23) = 3 \in K_{13} $, $ r(34) = 4 \in K_{24} $, and $ r(14) = 14 \in K_{1234} $. Again, invoking Theorem \ref{mainComplex}, we find that

\[
\wt{C}^{cell}_*(\mathbb{F}_a) \cong \KMW_2 \xr{(0,0)} \KMW_1 \oplus \KMW_1 \xr{0} \bZ.
\]

By summing the outcomes from the previous two cases, we obtain 
\[
\wt{C}^{cell}_*(\mathbb{F}_a) \cong \KMW_2 \xr{(a_{\epsilon}\eta,0)} \KMW_1 \oplus \KMW_1 \xr{0} \bZ.
\]
Using this complex, we can derive the Chow-Witt group of Hirzebruch surfaces, which is given by the following formulation:

\begin{equation*}
    \widetilde{\mathrm{CH}}^i(\mathbb{F}_a) =
    \begin{cases}
        \text{GW}(k) & \text{if } i=0 \text{ or } i=2 \text{ and } a \text{ is even;}\\
        2\Z \oplus \text{GW}(k) & \text{if } i=1 \text{ and } a \text{ is odd; }\\
        \text{GW}(k)^{\oplus 2} & \text{if } i=1 \text{ and } a \text{ is even; }\\
        \Z & \text{if } i=2 \text{ and } a \text{ is odd.}
    \end{cases}
\end{equation*}
\end{example}

In the subsequent analysis, we turn our attention to the general case of smooth complete toric surfaces. Utilizing the notation introduced in this section, the boundary complex $ K $ of the toric surface $ X_{\Sigma} $ possesses a shelling defined as follows:
\[
12, 2\underline{3}, 3\underline{4}, ..., (l-1)\underline{l}, \underline{1l}
\]
with appropriate marked restrictions. The characteristic function associated with this shelling is represented by $ \lambda: \bZ^l \to \bZ^2 $. The relations $ v_{i-1} + v_{i+1} = a_i v_i $ will generate the kernel of $ \lambda $. Leveraging the non-degeneracy of $ \lambda $ and the definition of $ \omega_{\kappa} $, we can consider $ \Lambda: \bZ_2^l \to \bZ_2^n $, which has the following kernel:

\begin{equation*}
  \begin{pmatrix}
    1 & 0 & \cdots & 0\\
    \chi(a_2)& 1 & \cdots &0\\
    1 & \chi(a_3) & \cdots &0 \\
    0 & 1 & \cdots &0\\
    \vdots& \vdots& \ddots &\vdots\\
    0 & 0 & \cdots & 1\\
    0 & 0 & \cdots & \chi(a_{l-1})\\
    0 & 0 & \cdots & 1\\
  \end{pmatrix}.
\end{equation*}
Now, if $ l $ is even and $ a_{\Sigma} = \mathrm{gcd}(a_1, \ldots, a_l) $ is even, it follows that $ \chi(a_i) = 0 $ for all $ i $. Under these conditions, $ \Lambda $ can be represented as follows:
\begin{equation*}
  \begin{pmatrix}
    1 & 0 & 1& 0 &\cdots &0\\
    0 & 1 & 0 & 1 &\cdots &1
  \end{pmatrix}.
\end{equation*}
In this scenario, $ \omega \in \{ \emptyset, 135\cdots(l-1), 246\cdots l, 123\cdots(l-1)l \} $, which implies that $ r(1l) = 1l \in K_{123\cdots(l-1)l} $ and identifies it as the only critical face contained in $ K_{123\cdots(l-1)l} $. Consequently, by Theorem \ref{mainComplex}, the cellular complex $ \wt{C}^{cell}_*(X_{\Sigma}) $ is represented as follows:

\[\begin{tikzcd}[ampersand replacement=\&,cramped]
	{\KMW_2} \& {\KMW_1} \\
	\& {(\KMW_1)^{\oplus l-3}} \& \bZ
	\arrow["0", from=1-1, to=1-2]
	\arrow["0"', from=1-1, to=2-2]
	\arrow["0", from=1-2, to=2-3]
	\arrow["\bigoplus"', draw=none, from=2-2, to=1-2]
	\arrow["0"', from=2-2, to=2-3]
\end{tikzcd}\]

On the other hand, if $ a_{\Sigma} = \mathrm{gcd}(a_1, \ldots, a_l) $ is odd, we may assume, without loss of generality, that at least one $ a_i $ is odd for $ 2 \leq i \leq l-1 $ and that $ i $ is odd. In this case, $ \Lambda $ can be represented by the following matrices:

\begin{equation*}
  \bordermatrix{%
   & 1 & 2 & \cdots & i-1 \cdots & l\cr
   & 1 & 0 & \cdots& 1  \cdots & 0\cr
   & 0 & 1 & \cdots& 1  \cdots & 1\cr
   }
\end{equation*}
or
\begin{equation*}
  \bordermatrix{%
   & 1 & 2 & \cdots & i+1 \cdots & l\cr
   & 1 & 0 & \cdots& 1  \cdots & 0\cr
   & 0 & 1 & \cdots& 1  \cdots & 1\cr
  }.
\end{equation*}

In both scenarios, $\omega$ cannot be expressed as $123\cdots(l-1)l$, and it follows that $r(1l)=1l\in K_{123\cdots(i-2)i\cdots(l-1)l}$ or $r(1l)=1l\in K_{123\cdots i(i+2)\cdots(l-1)l}$. Given that the reasoning is analogous, we can concentrate on the first case, which indicates the presence of two critical faces, namely $i$ and $1l$, in $K_{123\cdots(i-2)i\cdots(l-1)l}$. By employing similar reasoning as in Example \ref{Hirzebruch}, it can be established that these faces are connected by a differential in $\overline{C}^{\lambda}_i$. Consequently, we obtain the cellular complex $\wt{C}^{cell}_*(X_{\Sigma})$ represented as follows:

\[\begin{tikzcd}[ampersand replacement=\&,cramped]
	{\KMW_2} \& {\KMW_1} \\
	\& {(\KMW_1)^{\oplus l-3}} \& \bZ
	\arrow["{\eta}", from=1-1, to=1-2]
	\arrow["0"', from=1-1, to=2-2]
	\arrow["0", from=1-2, to=2-3]
	\arrow["\bigoplus"', draw=none, from=2-2, to=1-2]
	\arrow["0"', from=2-2, to=2-3]
\end{tikzcd}\]

When $l$ is odd, we assert that $a_{\Sigma}$ must also be odd, as stated in \cite[Theorem 7.5]{ewald1996combinatorial}, which posits that any smooth, complete toric surface can be successively transformed into either a Hirzebruch surface $\mathbb{F}_a$ (where $a\neq \pm 1$) or the projective plane $\mathbb{P}^2$. However, according to Examples \ref{P2} and \ref{blowdown}, if $l$ is odd, the surface can either be $\mathbb{P}^2$ or a blowup of another surface; in both instances, $a_\Sigma$ will be odd. Thus, we arrive at the complex $\wt{C}^{cell}_*(X_{\Sigma})$ described as follows:
\[\begin{tikzcd}[ampersand replacement=\&,cramped]
	{\KMW_2} \& {\KMW_1} \\
	\& {(\KMW_1)^{\oplus l-3}} \& \bZ
	\arrow["{\eta}", from=1-1, to=1-2]
	\arrow["0"', from=1-1, to=2-2]
	\arrow["0", from=1-2, to=2-3]
	\arrow["\bigoplus"', draw=none, from=2-2, to=1-2]
	\arrow["0"', from=2-2, to=2-3]
\end{tikzcd}\]

To summarize all the cases discussed above, we can express $\wt{C}^{cell}_*(X_{\Sigma})$ as

\[\begin{tikzcd}[ampersand replacement=\&,cramped]
	{\KMW_2} \& {\KMW_1} \\
	\& {(\KMW_1)^{\oplus l-3}} \& \bZ
	\arrow["{(a_{\Sigma})_{\epsilon}\eta}", from=1-1, to=1-2]
	\arrow["0"', from=1-1, to=2-2]
	\arrow["0", from=1-2, to=2-3]
	\arrow["\bigoplus"', draw=none, from=2-2, to=1-2]
	\arrow["0"', from=2-2, to=2-3]
\end{tikzcd}\]

By Proposition \ref{cellCohomology}, if we take $ M $ to be any $ \KMW_i $ or $ \KM_i $, the computations of the motivic cohomology groups, or the Milnor-Witt motivic cohomology groups, can be derived from the cellular complex of $ X $. For instance, utilizing the cellular complex of complete toric surfaces as discussed in Proposition \ref{examplesurface}, we obtain:

\begin{equation*}
        \widetilde{\mathrm{CH}}^i(X_{\Sigma})=
        \begin{cases}
            \text{GW}(k) & \text{if } i=0 \text{ or } i=2 \text{ and } a_{\Sigma} \text{ is even;}\\
            2\Z \oplus \text{GW}(k)^{\oplus l-3} & \text{if } i=1 \text{ and } a_{\Sigma} \text{ is odd;}\\
            \text{GW}(k)^{\oplus l-2} & \text{if } i=1 \text{ and } a_{\Sigma} \text{ is even;}\\
            \Z & \text{if } i=2 \text{ and } a_{\Sigma} \text{ is odd.}\\
        \end{cases}
\end{equation*}

In the context of general Milnor-Witt motivic cohomology, $ \text{GW}(k) $ can be substituted with $ \KMW_i $, $ 2\Z $ can be replaced by $ _{\eta}\KMW_i $, and $ \Z $ can be substituted with $ \KMW_i/\eta $. When we set $ M = \KM_i $, all differentials within the complex will vanish, thereby recovering the classical results regarding the Chow groups of toric varieties \cite{ewald1996combinatorial}.

\subsection{Exotic Examples}
\label{exotic}
In the subsequent examples, Theorem \ref{mainComplex} does not hold. Specifically, the cycle class map $ \mathrm{CH}^*(X_\Sigma) \to H^*(X_\Sigma(\mathbb{C}), \bZ) $ fails to be surjective. In fact, there is a contribution from the motivic cohomology groups (higher Chow groups) to $ H^*(X_\Sigma(\mathbb{C}), \bZ) $.

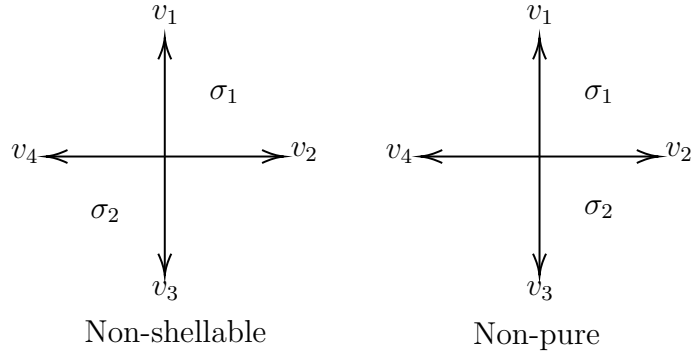
\begin{center}

\tikzset{every picture/.style={line width=0.75pt}} 

\begin{tikzpicture}[x=0.75pt,y=0.75pt,yscale=-1,xscale=1]

\draw    (120,110) -- (120,52) ;
\draw [shift={(120,50)}, rotate = 90] [color={rgb, 255:red, 0; green, 0; blue, 0 }  ][line width=0.75]    (10.93,-3.29) .. controls (6.95,-1.4) and (3.31,-0.3) .. (0,0) .. controls (3.31,0.3) and (6.95,1.4) .. (10.93,3.29)   ;
\draw    (120,110) -- (178,110) ;
\draw [shift={(180,110)}, rotate = 180] [color={rgb, 255:red, 0; green, 0; blue, 0 }  ][line width=0.75]    (10.93,-3.29) .. controls (6.95,-1.4) and (3.31,-0.3) .. (0,0) .. controls (3.31,0.3) and (6.95,1.4) .. (10.93,3.29)   ;
\draw    (120,110) -- (120,168) ;
\draw [shift={(120,170)}, rotate = 270] [color={rgb, 255:red, 0; green, 0; blue, 0 }  ][line width=0.75]    (10.93,-3.29) .. controls (6.95,-1.4) and (3.31,-0.3) .. (0,0) .. controls (3.31,0.3) and (6.95,1.4) .. (10.93,3.29)   ;
\draw    (120,110) -- (62,110) ;
\draw [shift={(60,110)}, rotate = 360] [color={rgb, 255:red, 0; green, 0; blue, 0 }  ][line width=0.75]    (10.93,-3.29) .. controls (6.95,-1.4) and (3.31,-0.3) .. (0,0) .. controls (3.31,0.3) and (6.95,1.4) .. (10.93,3.29)   ;
\draw    (309,110) -- (309,52) ;
\draw [shift={(309,50)}, rotate = 90] [color={rgb, 255:red, 0; green, 0; blue, 0 }  ][line width=0.75]    (10.93,-3.29) .. controls (6.95,-1.4) and (3.31,-0.3) .. (0,0) .. controls (3.31,0.3) and (6.95,1.4) .. (10.93,3.29)   ;
\draw    (309,110) -- (367,110) ;
\draw [shift={(369,110)}, rotate = 180] [color={rgb, 255:red, 0; green, 0; blue, 0 }  ][line width=0.75]    (10.93,-3.29) .. controls (6.95,-1.4) and (3.31,-0.3) .. (0,0) .. controls (3.31,0.3) and (6.95,1.4) .. (10.93,3.29)   ;
\draw    (309,110) -- (309,168) ;
\draw [shift={(309,170)}, rotate = 270] [color={rgb, 255:red, 0; green, 0; blue, 0 }  ][line width=0.75]    (10.93,-3.29) .. controls (6.95,-1.4) and (3.31,-0.3) .. (0,0) .. controls (3.31,0.3) and (6.95,1.4) .. (10.93,3.29)   ;
\draw    (309,110) -- (251,110) ;
\draw [shift={(249,110)}, rotate = 360] [color={rgb, 255:red, 0; green, 0; blue, 0 }  ][line width=0.75]    (10.93,-3.29) .. controls (6.95,-1.4) and (3.31,-0.3) .. (0,0) .. controls (3.31,0.3) and (6.95,1.4) .. (10.93,3.29)   ;

\draw (141,72.4) node [anchor=north west][inner sep=0.75pt]    {$\sigma _{1}$};
\draw (81,132.4) node [anchor=north west][inner sep=0.75pt]    {$\sigma _{2}$};
\draw (112,32.4) node [anchor=north west][inner sep=0.75pt]    {$v_{1}$};
\draw (182,101.4) node [anchor=north west][inner sep=0.75pt]    {$v_{2}$};
\draw (112,171.4) node [anchor=north west][inner sep=0.75pt]    {$v_{3}$};
\draw (41,102.4) node [anchor=north west][inner sep=0.75pt]    {$v_{4}$};
\draw (330,72.4) node [anchor=north west][inner sep=0.75pt]    {$\sigma _{1}$};
\draw (330,130.4) node [anchor=north west][inner sep=0.75pt]    {$\sigma _{2}$};
\draw (301,32.4) node [anchor=north west][inner sep=0.75pt]    {$v_{1}$};
\draw (371,101.4) node [anchor=north west][inner sep=0.75pt]    {$v_{2}$};
\draw (301,171.4) node [anchor=north west][inner sep=0.75pt]    {$v_{3}$};
\draw (230,102.4) node [anchor=north west][inner sep=0.75pt]    {$v_{4}$};
\draw (78,191) node [anchor=north west][inner sep=0.75pt]   [align=left] {Non-shellable};
\draw (274,192) node [anchor=north west][inner sep=0.75pt]   [align=left] {Non-pure};

\end{tikzpicture}
\end{center}
\begin{example}[Non-shellable]
For a fan represented on the left-hand side, we consider the associated cellular complex, which is depicted as follows:
\[\begin{tikzcd}[ampersand replacement=\&]
	{ (\KMW_2)_{\sigma_1}} \& { (\KMW_1 \otimes \mathbf{H})_{v_1}} \\
	\& { (\KMW_1 \otimes \mathbf{H})_{v_2}} \& {\mathbf{H} \otimes \mathbf{H}} \\
	{ (\KMW_2)_{\sigma_2}} \& { (\KMW_1 \otimes \mathbf{H})_{v_3}} \\
	\& { (\KMW_1 \otimes \mathbf{H})_{v_4}}
	\arrow["{(\mathrm{Id}, 0)}", from=1-1, to=1-2]
	\arrow["{(\mathrm{Id}, 0)}"', from=1-1, to=2-2]
	\arrow[from=1-2, to=2-3]
	\arrow["\bigoplus"{description}, draw=none, from=2-2, to=1-2]
	\arrow[from=2-2, to=2-3]
	\arrow["{(\mathrm{Id}, 0)}", from=3-1, to=3-2]
	\arrow["{(\mathrm{Id}, 0)}"'{pos=0.2}, from=3-1, to=4-2]
	\arrow["\bigoplus"{description}, draw=none, from=3-2, to=2-2]
	\arrow[from=3-2, to=2-3]
	\arrow["\bigoplus"{description}, draw=none, from=3-2, to=4-2]
	\arrow[from=4-2, to=2-3]
\end{tikzcd}\]
We find that this cellular complex is quasi-isomorphic to:
\[
  \wt{C}^{cell}_*(X_\Sigma) \cong  \KMW_2 \oplus \KMW_1\oplus \KMW_1\xr{0} \bZ
\]
In terms of MW-motive, we have 
\[  
  \Mt(X_\Sigma)\cong \tbZ(2)[3] \oplus \tbZ(1)[2] \oplus \tbZ(1)[2] \oplus \tbZ
\]
\end{example}
\begin{example}[Non-pure]
For a fan represented on the right-hand side, we consider the associated cellular complex, which is depicted as follows:
\[\begin{tikzcd}[ampersand replacement=\&,cramped]
	{ (\KMW_2)_{\sigma_1}} \& { (\KMW_1 \otimes \mathbf{H})_{v_1}} \\
	\& { (\KMW_1 \otimes \mathbf{H})_{v_2}} \& {\mathbf{H} \otimes \mathbf{H}} \\
	{ (\KMW_2)_{\sigma_2}} \& { (\KMW_1 \otimes \mathbf{H})_{v_3}} \\
	\& { (\KMW_1 \otimes \mathbf{H})_{v_4}}
	\arrow["{(\mathrm{Id}, 0)}", from=1-1, to=1-2]
	\arrow["{(\mathrm{Id}, 0)}"', from=1-1, to=2-2]
	\arrow[from=1-2, to=2-3]
	\arrow["\bigoplus"{description}, draw=none, from=2-2, to=1-2]
	\arrow[from=2-2, to=2-3]
	\arrow["{(\mathrm{Id}, 0)}"{pos=0.2}, from=3-1, to=2-2]
	\arrow["{(\mathrm{Id}, 0)}"', from=3-1, to=3-2]
	\arrow["\bigoplus"{description}, draw=none, from=3-2, to=2-2]
	\arrow[from=3-2, to=2-3]
	\arrow["\bigoplus"{description}, draw=none, from=3-2, to=4-2]
	\arrow[from=4-2, to=2-3]
\end{tikzcd}\]
Again we find that this complex is quasi-isomorphic to:
\[
  \wt{C}^{cell}_*(X_\Sigma) \cong  \KMW_1 \oplus \KMW_2 \oplus \KMW_1\xr{0} \bZ
\]
In terms of MW-motive, we have 
\[  
  \Mt(X_\Sigma)\cong \tbZ(1)[2] \oplus \tbZ(2)[3] \oplus \tbZ(1)[2] \oplus \tbZ
\]
\end{example}

\section*{Appendix}
\label{sec:appendix}
First, we observe that the morphism induced by $ \Gm \xr{(\cdot)^n} \Gm $ is given by
\[
n_{\epsilon}:\KMW_1\to \KMW_1, \quad \lambda \mapsto n_{\epsilon}\lambda, \quad n_{\epsilon}= \frac{n-\chi(n)}{2}h+ \chi(n) \in \KMW_0,
\]
where $ h = 1+\aBra{-1} $. We define $ h_{ij} := ( r_{ij} )_\epsilon $. It should be noted that $ n_{\epsilon} \eta = \chi(n)\eta $ and $ n_{\epsilon}[-1] = \chi(n)[-1] $.

For a subset $ \omega \subset \llBra{n} $, let $ \omega_{>j} \subset \omega $ denote the subset $ \{i\in \omega \mid i>j\} $. Similar notations apply for $ \omega_{<j}, \omega_{\leq j}, $ and $ \omega_{\geq j} $. 

\begin{lemma}
Let $ g_{\tau}: \Gm^{t_e}\to \Gm^{|\tau_e|} $ be a group section. Then, we have:
\begin{align*}
  g_{\tau*}[e] &= \sum_{\omega \subset \tau_e} \sum_{\substack{ f: \tauO_e \to \omega \\ I_j := f^{-1}(j) \\ S_j := f^{-1}(\omega_{<j}) }} \epsilon^{\mathrm{sgn}(f)} \eta^{t_e - |f(\tauO_e)|} (-[-1])^{|\omega|-|f(\tauO_e)|} \\ 
               & \prod_{j \in \omega \setminus f(\tauO_e)} \chi\left( \sum_{k \in S_j} r'_{kj} \right) \prod_{i \in \tauO_e} \aBra{-1}^{\sum_{j \in \omega_{>f(i)}} r'_{ij}} h'_{if(i)} [e_{\omega}].
\end{align*}
\end{lemma}

\begin{proof}

Now, we consider the action of $ g_\tau $ restricted to the coordinates of $ \tau $:
   \begin{align*}
     g_{\tau*} &: \prod_{i\in \tauO_e}[\lambda_i] \mapsto \prod_{i\in \tauO_e} \left( \bigotimes_{j\in \tau_e} (h_{ij} [\lambda_i]+1) \right) ([\lambda_i]+1) -1 \\ 
               &\mapsto {  \prod_{j \in \tau_e} }'\left( \bigotimes_{i\in \tauO_e} (h_{ij} [\lambda_i]+1) \right)\otimes \begin{cases}
                 ([\lambda_j]+1), &\text{ if } j\in \tauO_e\\
                 1, &\text{ if } j\in \tau^1_e\\
               \end{cases}  \\ 
& \text{(Let ${ \prod}' $ sum only those monomials which include all $[\lambda_i]$ terms after expanding the product)}\\
            &\mapsto {  \prod_{j \in \tau_e} }' \left( \bigotimes_{i\in \tauO_e} (h_{ij} [\lambda_i]+1) \right)+ \begin{cases}
              \left( \prod_{i \in \tauO_e} ( h_{ij}\eta [\lambda_i]+1) \right)[\lambda_j], &\text{ if } j\in \tauO_e\\
                 0, &\text{ if } j\in \tau^1_e\\
               \end{cases}   \\
            &= {  \prod_{j \in \tau_e} }' 1+\left[ \prod_{i \in \tauO_e} \lambda_i^{r_{ij}} \right] + \begin{cases}
              \left( \prod_{i \in \tauO_e} \aBra{\lambda_i}^{r_ij} \right)[\lambda_j], &\text{ if } j\in \tauO_e\\
                 0, &\text{ if } j\in \tau^1_e\\
               \end{cases}   \\
            &= \sum_{\omega \subset \tau_e} { \prod_{j \in \omega} }'  \begin{cases}
                \left[ ( \prod_{i \in \tauO_e} \lambda_i^{r_{ij}} )\lambda_j \right] , &\text{ if } j\in \tauO_e\\
                 \left[ \prod_{i \in \tauO_e} \lambda_i^{r_{ij}} \right], &\text{ if } j\in \tau^1_e\\
               \end{cases}   \\
            & \text{(Let $r'_{ij}=r_{ij}+\delta_{ij}, h'_{ij}=(r'_{ij})_\epsilon $)}\\  
            &=  \sum_{\omega \subset \tau_e} { \prod_{j \in \omega} }' \left[ \prod_{i \in \tauO_e} \lambda_i^{r'_{ij}} \right]\\
            &=  \sum_{\omega \subset \tau_e} \sum_{ \emptyset \neq I_1 \subset \tauO_e} \eta^{|I_1|-1}  \prod_{i \in I_1} h'_{ij_1}[\lambda_i]{ \prod_{j \in \omega \setminus \{j_1\}} }' \left[ (-1)^{\sum_{k\in I_1} r'_{kj}}\prod_{i \in \tauO_e \setminus I_1} \lambda_i^{r'_{ij}} \right]\\
            &=  \sum_{\omega \subset \tau_e} \sum_{ \emptyset \neq I_1 \subset \tauO_e} \eta^{|I_1|-1}  \prod_{i \in I_1} h'_{ij_1}[\lambda_i] \left( \sum_{  I_2 \subset \tauO_e \setminus I_1} \aBra{ -1 }^{\sum_{k\in I_1} r'_{kj_2}}\eta^{|I_2|-1}  \prod_{i \in I_2} h'_{ij_2}[\lambda_i]\right) \\ 
            & { \prod_{j \in \omega \setminus \{j_1,j_2\}} }' \left[ (-1)^{\sum_{k\in I_1 \sqcup I_2} r'_{kj}}\prod_{i \in \tauO_e \setminus ( I_1\sqcup I_2 )} \lambda_i^{r'_{ij}} \right]\\
            & \text{(For the case $I_2=\emptyset$, we formally define $\aBra{ -1 }^{\sum_{k\in I_1} r'_{kj}}\eta^{-1}:= [(-1)^{\sum_{k\in I_1} r'_{kj}}]$)} \\
            &=  \sum_{\omega \subset \tau_e} \sum_{\substack{ f: \tauO_e \to \omega  \\ I_j:=f^{-1}(j) \\ S_j:=f^{-1}(\omega_{<j}) }} \eta^{t_e - |f(\tauO_e)| }\prod_{j \in \omega  } \begin{cases}
              \aBra{ -1 }^{\sum_{k\in S_j} r'_{kj}}\prod_{i \in I_j} h'_{ij}[\lambda_i], &\text{ if } |I_j|\neq0\\
              -\aBra{ -1 }^{\sum_{k\in S_j} r'_{kj}}[(-1)^{\sum_{k\in S_j} r'_{kj}}], &\text{ if } |I_j|=0\\
            \end{cases}   \\
            &=  \sum_{\omega \subset \tau_e} \sum_{\substack{ f: \tauO_e \to \omega  \\ S_j:=f^{-1}(\omega_{<j}) }} \epsilon^{ \mathrm{sgn}(f)} \eta^{ t_e - |f(\tauO_e)| } (-[-1])^{|\omega|-|f(\tauO_e)|} \aBra{ -1 }^{\sum_{j \in  \omega}\sum_{k\in S_j} r'_{kj}}\\ 
  &  \prod_{j \in \omega \setminus f(\tauO_e)} \chi( \sum_{k\in S_j} r'_{kj} ) \prod_{i \in \tauO_e} h'_{if(i)} \prod_{i\in \tauO_e} [\lambda_i] \\ 
   & \text{(Notice that $ \sum_{j \in  \omega}\sum_{k\in S_j} r'_{kj}=\sum_{k\in \tauO_e} \sum_{\substack{j \in  \omega\\ \ j>f(k)} } r'_{kj}$) }\\
   &=  \sum_{\omega \subset \tau_e} \sum_{\substack{ f: \tauO_e \to \omega  \\ S_j:=f^{-1}(\omega_{<j}) }} \epsilon^{ \mathrm{sgn}(f)} \eta^{ t_e - |f(\tauO_e)| } (-[-1])^{|\omega|-|f(\tauO_e)|} \\ 
   &  \prod_{j \in \omega \setminus f(\tauO_e)} \chi( \sum_{k\in S_j} r'_{kj} ) \prod_{i \in \tauO_e}\aBra{-1}^{\sum_{j \in  \omega_{>f(i)}   } r'_{ij}} h'_{if(i)} \prod_{i\in \tauO_e} [\lambda_i]  
   \end{align*}
Here, $\mathrm{sgn}(f)= \chi\left( |\{(i_1, i_2) \in (\tauO_e)^2 \mid i_1<i_2, f(i_1)>f(i_2)\}| \right)$. When $f$ is a bijection, $\mathrm{sgn}(f)$ corresponds to the standard sign of a permutation between ordered sets. Consequently, the coefficient $\epsilon^{ \mathrm{sgn}(f)}$ arises from the graded commutative structure of $\KMW_*$. Furthermore, $\prod_{i \in \tauO_e} h'_{if(i)}$ can be simplified to $\prod_{i \in \tauO_e} \chi( r'_{if(i)} )$, except in cases where $f$ is a bijection.

Finally, we obtain the action:
\begin{align*}
  g_{\tau*}[e] &= \sum_{\omega \subset \tau_e} \sum_{\substack{ f: \tauO_e \to \omega   \\ S_j:=f^{-1}(\omega_{<j}) }} \epsilon^{ \mathrm{sgn}(f)} \eta^{ t_e - |f(\tauO_e)| } (-[-1])^{|\omega|-|f(\tauO_e)|} \\ 
            &  \prod_{j \in \omega \setminus f(\tauO_e)} \chi( \sum_{k\in S_j} r'_{kj} ) \prod_{i \in \tauO_e}\aBra{-1}^{\sum_{j \in  \omega_{>f(i)}  } r'_{ij}} h'_{if(i)}  [e_{\omega}] 
\end{align*}
\end{proof}

Now, we can derive a general result regarding the action of $ g $:

\begin{proposition}
  \label{tActCompute}
  Let $ g: \Gm^{t_e} \to \Gm^{n} $ be a group section. The action induced by $ g $ is given by $ g_*[e] = \sum_{\omega \subset \tau_e} c(g,e)_\omega [e_{\omega}] $, where
  \begin{align*}
    c(g,e)_\omega &=  \sum_{\substack{ f: \tauO_e \to \omega \sqcup \{\infty\}  \\ S_j:=f^{-1}(\omega_{<j}) }} \epsilon^{ \mathrm{sgn}(f)} \eta^{ t_e - |f(\tauO_e)_{<\infty} | } ( -[-1] )^{|\omega|-|f(\tauO_e)_{<\infty}|} \\ 
  &  \prod_{j \in \omega \setminus f(\tauO_e)} \chi( \sum_{k\in S_j} r'_{kj} )  \prod_{i \in \tauO_e} \aBra{-1}^{\sum_{j \in (\omega \sqcup \{\infty\})_{>f(i)} } r'_{ij}} h'_{if(i)}  
  \end{align*}
  
  Additionally, we have 
  \[
    c(g,e)_\omega \in \begin{cases}
      \bZ \eta^{t_e - |\omega|}, & \text{if } t_e > |\omega|\\
      \bZ + \bZ h, & \text{if } t_e = |\omega|\\
      \bZ [-1]^{|\omega| - t_e}, & \text{if } t_e < |\omega|\\
    \end{cases} 
    \subset \KMW_{|\omega| - t_e}
  \]
\end{proposition}
\begin{proof}
  Utilizing a similar methodology, we first analyze the action of $g_*$ on $\KMW_*$. Without loss of generality, we can assume that $ \sigma_e = \llBra{|e|}$. We have the following action:
    \begin{align*}
      g_{*} &: \prod_{i\in \sigma_e \sqcup \tauO_e  }[\lambda_i] \mapsto \prod_{k\in \sigma_e} [\lambda_k]\prod_{i\in \tauO_e} \left( \bigotimes_{j\in \tau_e} (h_{ij} [\lambda_i]+1) \right) ([\lambda_i]+1) -1 \\ 
            &\mapsto {  \prod_{j \in \llBra{n}} }'\left( \bigotimes_{i\in \tauO_e} (h_{ij} [\lambda_i]+1) \right)\otimes \begin{cases}
                 ([\lambda_j]+1), &\text{ if } j\in \tauO_e\\
                 1, &\text{ if } j\in \tau^1_e\\
                 [\lambda_j], &\text{ if } j\in \sigma_e\\
               \end{cases}  \\ 
            &\mapsto {  \prod_{j \in \llBra{n}} }'  \begin{cases}
              \left( \prod_{i \in \tauO_e} \aBra{\lambda_i}^{r_{ij}} \right)[\lambda_j], &\text{ if } j\in \sigma_e\\
                 1+ \left[ \prod_{i \in \tauO_e} \lambda_i^{r'_{ij}} \right], &\text{ if } j\in \tau_e\\
               \end{cases}   \\
            &=  \sum_{\omega \subset \tau_e} \prod_{k\in \sigma_e} [\lambda_k]{ \prod_{j \in \omega} }' \aBra{ \prod_{i \in \tauO_e} \lambda_i^{\chi( \sum_{k\in \sigma_e} r'_{ik} )}} \left[ \prod_{i \in \tauO_e} \lambda_i^{r'_{ij}} \right]\\
            &=  \sum_{\omega \subset \tau_e} \prod_{k\in \sigma_e} [\lambda_k]{ \prod_{j \in \omega} }' \sum_{ \emptyset \neq I_j \subset \tauO_e} \aBra{ \prod_{i \in \tauO_e} \lambda_i^{\chi( \sum_{k\in \sigma_e} r'_{ik} )}}\eta^{|I_j|-1}  \prod_{i \in I_j} h'_{ij}[\lambda_i]\\           
            &=  \sum_{\omega \subset \tau_e} \prod_{k\in \sigma_e} [\lambda_k]\sum_{\substack{ f: \tauO_e \to \omega \sqcup \{ \infty \}  \\ I_j:=f^{-1}(j) \\ S_j:=f^{-1}(\omega_{<j}) }} \eta^{t_e - |f(\tauO_e)_{<\infty}| }  \aBra{-1}^{ \sum_{i\in S_\infty} \sum_{k\in \sigma_e} r'_{ ik }}\\ 
          &\left( \prod_{i\in I_\infty}h'_{i\infty}[\lambda_i] \right)\prod_{j \in \omega  } \begin{cases}
       \aBra{ -1 }^{\sum_{k\in S_j} r'_{kj}}\prod_{i \in I_j} h'_{ij}[\lambda_i], &\text{ if } |I_j|\neq0\\
       -\aBra{ -1 }^{\sum_{k\in S_j} r'_{kj}} [(-1)^{\sum_{k\in S_j} r'_{kj}}], &\text{ if } |I_j|=0\\
        \end{cases}   \\
          & \text{(Let $r_{i\infty}=r'_{i\infty}=\sum_{k\in \sigma_e}r'_{ik}$ and $h_{i\infty}=r_{i\infty \epsilon}$, and notice that $ i\in S_j \Leftrightarrow j \in \omega\sqcup \{\infty\}, j>f(i) $)} \\
          &=  \sum_{\omega \subset \tau_e} \sum_{\substack{ f: \tauO_e \to \omega \sqcup \{\infty\}  \\ S_j:=f^{-1}(\omega_{<j}) }} \epsilon^{ \mathrm{sgn}(f)} \eta^{ t_e - |f(\tauO_e)_{<\infty} | } ( -[-1] )^{|\omega|-|f(\tauO_e)_{<\infty}|} \\ 
  &  \prod_{j \in \omega \setminus f(\tauO_e)} \chi( \sum_{k\in S_j} r'_{kj} )  \prod_{i \in \tauO_e}\aBra{-1}^{\sum_{j \in  (\omega \sqcup \{\infty\})_{>f(i)}  } r'_{ij}} h'_{if(i)} \prod_{i\in \sigma_e \sqcup \tauO_e} [\lambda_i] \\
   \end{align*}
This leads to the expression: 
  \begin{align*}
    g_*[e] &= \sum_{\omega \subset \tau_e} \sum_{\substack{ f: \tauO_e \to \omega \sqcup \{\infty\}  \\ S_j:=f^{-1}(\omega_{<j}) }} \epsilon^{ \mathrm{sgn}(f)} \eta^{ t_e - |f(\tauO_e)_{<\infty} | } ( -[-1] )^{|\omega|-|f(\tauO_e)_{<\infty}|} \\ 
  &  \prod_{j \in \omega \setminus f(\tauO_e)} \chi( \sum_{k\in S_j} r'_{kj} )  \prod_{i \in \tauO_e}\aBra{-1}^{\sum_{j \in  (\omega \sqcup \{\infty\})_{>f(i)}  } r'_{ij}} h'_{if(i)} [e_{\omega}] 
  \end{align*}
  
Finally, we observe that
\[
\eta[-1] \eta = -2 \eta,\ \eta[-1] [-1] = -2 [-1],\ \eta[-1] = h-2,\ \aBra{-1} = h-1,
\]
and that $ n_{\epsilon} \eta = \chi(n)\eta $, $ n_{\epsilon}[-1] = \chi(n)[-1] $. Consequently, we derive that
  \[
    c(g,e)_\omega \in \begin{cases}
      \bZ \eta^{t_e - |\omega|}, &\text{ if }t_e > |\omega|\\
      \bZ + \bZ h, &\text{ if }t_e = |\omega|\\
      \bZ [-1]^{|\omega| - t_e}, &\text{ if }t_e < |\omega|\\
    \end{cases} 
    \subset \KMW_{|\omega| - t_e}
  \]
\end{proof}

\begin{corollary}
   Let $\omega \subset \tau_e$ such that $t_e = |\omega|$. In this case, we have the following expression for $c(g, e)_\omega$:
   \[
   c(g,e)_\omega=   \frac{1}{2}\left( \mathrm{det}( [r'_{ij}]_{i \in \tauO_e}^{ j\in \omega})h  -\eta[-1]\sum_{\pi \subset \omega }  (-1)^{|\omega|-|\pi|}  \prod_{i \in \tauO_e} \chi( \sum_{j \in \omega  \sqcup \sigma_e} r'_{ij} ) \right) 
   \]

   Furthermore, if $t_e > |\omega|$, then the expression for $c(g, e)_\omega$ becomes:
   \[
    c(g,e)_\omega= \eta^{ t_e - |\omega| }\sum_{\pi \subset \omega }  (-1)^{|\omega|-|\pi|}  \prod_{i \in \tauO_e} \chi( \sum_{j \in \omega \sqcup \sigma_e} r'_{ij} )   
   \]

   Conversely, if $t_e < |\omega|$, the expression is given by:
   \[
     c(g,e)_\omega= [-1]^{ |\omega| - t_e}(-2)^{ t_e - |\omega|  }\sum_{\pi \subset \omega }  (-1)^{|\omega|-|\pi|}  \prod_{i \in \tauO_e} \chi( \sum_{j \in \omega \sqcup \sigma_e} r'_{ij} )   
   \]
\end{corollary}

\begin{proof}
Notice that $(a+b)_\epsilon = a_\epsilon + \aBra{-1}^{a} b_\epsilon$. If $f(i) \in P$, then the factor $\aBra{-1}^{\sum_{j \in P_{>f(i)}} r'_{ij}} h'_{if(i)}$ can be expressed as $(\sum_{j \in P_{\geq f(i)}} r'_{ij})_\epsilon - (\sum_{j \in P_{>f(i)}} r'_{ij})_\epsilon$. As a result, the factor $c(g,e)_\omega$, under the condition that $t_e \geq |\omega|$, can be simplified as follows:

  \begin{align*} 
    & c(g,e)_\omega \prod_{i\in  \sigma_e \sqcup \tauO_e} [\lambda_i] 
    = \prod_{k\in \sigma_e} [\lambda_k]\sum_{\substack{ f: \tauO_e \to \omega \sqcup \{ \infty \}  \\ I_j:=f^{-1}(j) \\ S_j:=f^{-1}(\omega_{<j}) }} \eta^{t_e - |\omega| }  \aBra{-1}^{ \sum_{i\in S_\infty} \sum_{k\in \sigma_e} r'_{ ik }}\\ 
          &\left( \prod_{i\in I_\infty}h'_{i\infty}[\lambda_i] \right)\prod_{j \in \omega  } \begin{cases}
              \aBra{ -1 }^{\sum_{k\in S_j} r'_{kj}}\prod_{i \in I_j} h'_{ij}[\lambda_i], &\text{ if } |I_j|\neq0\\
              \eta[(-1)^{\sum_{k\in S_j} r'_{kj}}]= \aBra{ -1 }^{\sum_{k\in S_j} r'_{kj}}-1, &\text{ if } |I_j|=0\\
            \end{cases}   \\
          &=   \sum_{\substack{ f: \tauO_e \to \omega \sqcup \{ \infty \}\\ S_j:=f^{-1}(\omega_{<j}) }}\sum_{\pi \subset \omega \setminus f(\tauO_e)} \epsilon^{ \mathrm{sgn}(f)} \eta^{ t_e - |\omega| } (-1)^{|\pi|} \aBra{ -1 }^{\sum_{j \in  (\omega \setminus \pi) \sqcup \{\infty\}}\sum_{k\in S_j} r'_{kj}} \prod_{i \in \tauO_e} h'_{if(i)} \prod_{i\in \sigma_e \sqcup \tauO_e} [\lambda_i] \\ 
   & \text{(Notice that $ \sum_{j \in  (\omega \setminus \pi) \sqcup \{\infty\}}\sum_{k\in S_j} r'_{kj}=\sum_{k\in \tauO_e} \sum_{\substack{j \in  (\omega \setminus \pi) \sqcup \{\infty\}\\ \ j>f(k)} } r'_{kj}$)}\\
            &=   \sum_{ f: \tauO_e \to \omega \sqcup \{\infty\} }\sum_{\pi \subset \omega \setminus f(\tauO_e)} \epsilon^{ \mathrm{sgn}(f)} \eta^{ t_e - |\omega| } (-1)^{|\pi|}  \prod_{i \in \tauO_e} \aBra{-1}^{\sum_{j \in  ( (\omega \setminus \pi) \sqcup \{\infty\})_{>f(i)}   } r'_{ij}}h'_{if(i)} \prod_{i\in \sigma_e \sqcup \tauO_e} [\lambda_i] \\ 
            &=   \eta^{ t_e - |\omega| }\sum_{\pi \subset \omega }\sum_{ f: \tauO_e \to ( \omega \setminus \pi )  \sqcup \{\infty\}} \epsilon^{ \mathrm{sgn}(f)} (-1)^{|\pi|}  \prod_{i \in \tauO_e} \aBra{-1}^{\sum_{j \in  ( (\omega \setminus \pi) \sqcup \{\infty\})_{>f(i)}   } r'_{ij}}h'_{if(i)} \prod_{i\in \sigma_e \sqcup \tauO_e} [\lambda_i] \\  
            &\text{(And if $t_e>|\omega|$, then $ \eta^{ t_e - |\omega| }\epsilon = \eta^{ t_e - |\omega| }$)}\\
            &=   \eta^{ t_e - |\omega| }\sum_{\pi \subset \omega }  (-1)^{|\pi|}  \prod_{i \in \tauO_e} \left( \sum_{f(i) \in ( \omega \setminus \pi ) \sqcup \{\infty\}} \aBra{-1}^{\sum_{j \in  ( ( \omega \setminus \pi ) \sqcup \{\infty\})_{>f(i)}   } r'_{ij}}h'_{if(i)} \right)\prod_{i\in \sigma_e \sqcup \tauO_e} [\lambda_i] \\ 
            &=   \eta^{ t_e - |\omega| }\sum_{\pi \subset \omega }  (-1)^{|\pi|}  \prod_{i \in \tauO_e} ( \sum_{j \in ( \omega \setminus \pi ) \sqcup \sigma_e} r'_{ij} )_\epsilon \prod_{i\in  \sigma_e \sqcup \tauO_e} [\lambda_i] \\ 
            &=   \eta^{ t_e - |\omega| }\sum_{\pi \subset \omega }  (-1)^{|\pi|}  \prod_{i \in \tauO_e} \chi( \sum_{j \in ( \omega \setminus \pi ) \sqcup \sigma_e} r'_{ij} ) \prod_{i\in  \sigma_e \sqcup \tauO_e} [\lambda_i]  \\
   \end{align*}
If $t_e = |\omega|$, we observe from Proposition \ref{tActCompute} that $c(g,e)_{\omega} \in \mathbb{Z} \oplus \mathbb{Z}h$ has no $2$-torsion, and we have the following identities: $h \aBra{-1} = h$, $h \epsilon = -h$, and $h n_\epsilon = h n$.
   \begin{align*}
            & 2c(g,e)_{\omega}= (h-\eta[-1])c(g,e)_\omega = -\eta[-1]\sum_{\pi \subset \omega }  (-1)^{|\pi|}  \prod_{i \in \tauO_e} \chi( \sum_{j \in ( \omega \setminus \pi ) \sqcup \sigma_e} r'_{ij} )   \\
            &+h\sum_{\pi \subset \omega }  (-1)^{|\pi|}    \sum_{\substack{  f: \tauO_e \to \omega \setminus \pi  \sqcup \{\infty\} }} (-1)^{\mathrm{sgn}(f)}  \prod_{i \in \tauO_e} r'_{if(i)}  \\ 
            &= h\sum_{\substack{  f: \tauO_e \to \omega \\ f \text{ is bijective} }} (-1)^{\mathrm{sgn}(f)}  \prod_{i \in \tauO_e} r'_{if(i)}  -\eta[-1]\sum_{\pi \subset \omega }  (-1)^{|\pi|}  \prod_{i \in \tauO_e} \chi( \sum_{j \in ( \omega \setminus \pi ) \sqcup \sigma_e} r'_{ij} ) \\ 
            &=  \mathrm{det}( [r'_{ij}]_{i \in \tauO_e}^{ j\in \omega})h  -\eta[-1]\sum_{\pi \subset \omega }  (-1)^{|\pi|}  \prod_{i \in \tauO_e} \chi( \sum_{j \in ( \omega \setminus \pi ) \sqcup \sigma_e} r'_{ij} ) \\ 
   \end{align*}
Similarly, observe that $ 2[(-1)^a] = -\eta[(-1)^a][-1] = (-[-1])(\aBra{-1}^a - 1) $. For $ t_e < |\omega| $, we have:
   \begin{align*}
    & 2^{|\omega|-t_e}c(g,e)_\omega \prod_{i\in  \sigma_e \sqcup \tauO_e} [\lambda_i] \\
    &= \prod_{k\in \sigma_e} [\lambda_k]\sum_{\substack{ f: \tauO_e \to \omega \sqcup \{ \infty \}  \\ I_j:=f^{-1}(j) \\ S_j:=f^{-1}(\omega_{<j}) }} (-[-1])^{ |\omega| - t_e}  \aBra{-1}^{ \sum_{i\in S_\infty} \sum_{k\in \sigma_e} r'_{ ik }}\\ 
          &\left( \prod_{i\in I_\infty}h'_{i\infty}[\lambda_i] \right)\prod_{j \in \omega  } \begin{cases}
              \aBra{ -1 }^{\sum_{k\in S_j} r'_{kj}}\prod_{i \in I_j} h'_{ij}[\lambda_i], &\text{ if } |I_j|\neq0\\
              \eta[(-1)^{\sum_{k\in S_j} r'_{kj}}]= \aBra{ -1 }^{\sum_{k\in S_j} r'_{kj}}-1, &\text{ if } |I_j|=0\\
            \end{cases}   \\
            &=   (-[-1])^{ |\omega| - t_e}\sum_{\pi \subset \omega }  (-1)^{|\pi|}  \prod_{i \in \tauO_e} \chi( \sum_{j \in (\omega \setminus \pi) \sqcup \sigma_e} r'_{ij} ) \prod_{i\in  \sigma_e \sqcup \tauO_e} [\lambda_i]  
   \end{align*}
Again from Proposition \ref{tActCompute}, we have $ c(g,e)_\omega = n_{g,e,\omega} [-1]^{|\omega| - t_e} $, which indicates that it has no $ 2^p $-torsion.
\end{proof}
\bibliographystyle{plain}
\bibliography{references}
\end{document}